\documentclass[english]{amsart}

\usepackage{amsmath, enumerate}

\usepackage{amssymb}

\usepackage{amscd} \usepackage{tabularx}
\usepackage[all,cmtip,line]{xy}

\newcommand{\F}{\mathbb{F}} 
\def\RCS$#1: #2 ${\expandafter\def\csname RCS#1\endcsname{#2}}
\RCS$Revision: 280 $
\RCS$Date: 2009-12-17 22:28:17 -0500 (Thu, 17 Dec 2009) $
 
\DeclareMathOperator{\Frob}{Frob}

\newcommand{\labelbp}[1]
{\refstepcounter{equation}\hfill(\theequation)\label{#1}}

 \newcommand{\eps}{\epsilon}

\newcommand{\To}{\longrightarrow}\newcommand{\into}{\hookrightarrow}

\newcommand{\onto}{\twoheadrightarrow}
\newcommand{\isoto}{\stackrel{\sim}{\To}}

\newcommand{\cO}{\mathcal{O}}
\newcommand{\bigO}{\mathcal{O}} 
\newcommand{\cSL}{\mathcal{SL}}

\newcommand{\cG}{\mathcal{G}}
\newcommand{\cV}{\mathcal{V}}
\newcommand{\cT}{\mathcal{T}}
\newcommand{\cD}{\mathcal{D}}
\newcommand{\cB}{\mathcal{B}}
\newcommand{\cU}{\mathcal{U}}
\newcommand{\cH}{\mathcal{H}}
\newcommand{\Pj}{\mathbb{P}}

\newcommand{\Z}{\mathbb{Z}} 
\newcommand{\Q}{\mathbb{Q}}

\newcommand{\C}{\mathbb{C}}

\newcommand{\otimesind}{\otimes\text{-}\!\Ind}

\DeclareMathOperator{\Stab}{Stab}

\newcommand{\Gal}{\operatorname{Gal}}
\newcommand{\GO}{\operatorname{GO}}
\newcommand{\GL}{\operatorname{GL}}

\newcommand{\Sp}{\operatorname{Sp}}

\newcommand{\PGL}{\operatorname{PGL}}

 \newcommand{\Qbar}{\overline{\Q}}

\newcommand{\Ind}{\operatorname{Ind}}
\newcommand{\SL}{\operatorname{SL}}

\newcommand{\ad}{\operatorname{ad}}

\newcommand{\Hom}{\operatorname{Hom}}

\newcommand{\Aut}{\operatorname{Aut}}

\newcommand{\mar}[1]{\marginpar{\tiny #1}}                                 
 
\newcommand{\Gm}{{\mathbb{G}_m}}

\usepackage{amsthm}

\newtheorem*{claim}{Claim}  
\newtheorem{thm}{Theorem}[subsection]
\newtheorem{corollary}[thm]{Corollary}
\newtheorem{cor}[thm]{Corollary}
 \newtheorem{lemma}[thm]{Lemma}
\newtheorem{lem}[thm]{Lemma} \newtheorem{prop}[thm]{Proposition}
 \theoremstyle{definition}
 \theoremstyle{definition}
\newtheorem{defn}[thm]{Definition} \theoremstyle{remark}
\newtheorem{rem}[thm]{Remark} 
\newtheorem{remark}[thm]{Remark} 
 
\numberwithin{equation}{subsection}

\theoremstyle{definition}

\begin{document}
\title{Non-big subgroups for $l$ large}

\author{Thomas Barnet-Lamb}\email{tbl@brandeis.edu}\address{Department of Mathematics, Brandeis University, 415 South Street MS 050, Waltham MA}
\begin{abstract} Lifting theorems form an important collection of tools in showing 
that Galois representations are
associated to automorphic forms. (Key examples in dimension $n>2$ are
the lifting theorems of Clozel, Harris and Taylor and of Geraghty.) All present lifting 
theorems for $n>2$ dimensional representations have a certain rather technical 
hypothesis---the residual image must be `big'.
The aim of this paper is to demystify this condition somewhat.

For a fixed integer $n$, and a prime $l$ larger than a constant depending on $n$, we 
show that $n$ dimensional mod $l$ representations which fail to be big must be of one 
of three kinds: they either
fail to be absolutely irreducible, are induced from representations of larger fields, or can be
written as a tensor product including a factor which is the reduction of an Artin 
representation in characteristic zero. Hopefully this characterization will make the bigness
condition more comprehensible, at least for large $l$.
\end{abstract}
\maketitle
\tableofcontents
\section{Introduction}

\subsection{}
Recent years have seen some progress in proving modularity lifting theorems 
and potential modularity theorems for Galois representations of dimension $n>2$: 
for instance, see \cite{cht,hsbt,ger,BLGHT,BLGG,BLGGT}. One hypothesis in common to all these
theorems is that the image of the residual representation to which we want to
apply the theorem must be `big', a rather technical hypothesis first introduced 
by Clozel, Harris and Taylor in \cite{cht}. (See the beginning of \S2 for a full 
statement of this condition.) In fact, several of these theorems rely on a stronger 
hypothesis, introduced in \cite{BLGHT}: the image of the residual representation must
be `$M$-big'. (Again, see \S2 for a full statement of this condition. We remark that
generally one writes `$m$-big' rather than `$M$-big', but we suffer from a shortage
of letters in this paper and have had to economize.)

While on the one hand it seemed to be fairly easy to show that these bigness and $M$-bigness
hypotheses hold with a specific example in mind (see, for instance Corollaries 2.4.3 and 2.5.4
and Lemmas 2.5.5 and 2.5.6 of \cite{cht} or Lemmas 7.3--4 of \cite{BLGHT}), it remained for
some time the case that, considered in general, the condition was rather 
mysterious---especially given the fact that its precise definition is so technically involved.

A considerable advance in the situation came with the paper \cite{snowden2009bigness}.
The main result of this paper was to investigate the question of bigness in compatible
systems, showing that for compatible systems of Galois representations which are `strongly 
irreducible' (in the sense that for any given finite field extension, the restrictions of the 
representations in the compatible system to that extension are irreducible at almost all
places), the residual image will be big at a positive Dirichlet density of places. Along the way,
they prove that a large class of subgroups are big (in particular, the images of 
certain algebraic representations), thus giving a rather general situation in which 
bigness can be deduced (for instance,  the specific examples considered in \cite{cht,BLGHT}, 
are much more specific and are indeed immediate consequences). This was generalized to
$M$-bigness by White in \cite{pjw}.

The aim of this paper is to push the powerful techniques introduced by Snowden and
Wiles in that paper a little further, combining them with a little group theory to give
a still larger class of representations with big image. The class is rather broad, and 
indeed the only representations \emph{not} in the class are of rather special kinds.

Our main theorem is the following (see Theorem \ref{main-theorem again}):
\begin{thm} \label{main-thm}
For each pair of positive integers $n$ and $M$, there is an 
integer $C(M,n)$ with the following property. Let $l>C(M,n)$ be a prime, $k/\F_l$ a finite 
extension with $l\nmid[k:\F_l]$, 
$\Gamma_0$ a profinite group, and $r:\Gamma_0\to\GL_n(k)$ a representation. 
For convenience of notation, let us choose a number field $L$ and prime $\lambda$ 
of $L$ such that $\bigO_L/\lambda \bigO_L=k$. Suppose that the image of $r$ is not 
$M$-big. Then one of the following must hold:
\begin{enumerate}
\item $r$ does not act absolutely irreducibly on $k^n$,
\item there is a proper subgroup $\Gamma_0'<\Gamma_0$, an integer $m|n$, and
a representation $r':\Gamma_0'\to \GL_m(k)$ such that 
$r=\Ind_{\Gamma_0'}^{\Gamma_0} r'$, or
\item there are representations 
$r_1:\Gamma_0\to\GL_m(\bigO_L)$ and
$r_2:\Gamma_0\to\GL_{m'}(k)$, with open kernels and with $m>1$ and $n=mm'$,
such that $r= \bar{r}_1\otimes r_2$.
\end{enumerate}
Furthermore, in case (3) the order of the image of $r_1$ can be bounded in terms
of $n$.
\end{thm}

\begin{remark} It is natural to ask whether the constant $C(m,n)$ is effective. The
answer is that, in principle, a sufficiently assiduous study of the proofs here and
in the various papers cited here should allow one to determine effective bounds on
$C(m,n)$. (The same holds for the similar constants $D(n,N,a)$ and $E(n,m,a)$ below.)
On the other hand, one would expect these bounds to be rather large, 
and so we will not attempt this arduous calculation here.
\end{remark}

\begin{remark} The condition that $l\nmid[k:\F_l]$ is probably harmless, but in case
it were ever problematic, it is worth mentioning that it is effectively dispensible. In 
particular, one of the requirements of a subgroup $\Gamma<\GL_n(k)$ being ($M$-)big is 
that it have no $l$ power order quotient. In all applications we know of, however,
it suffices to have the weaker property that any normal subgroup $\Gamma'$ of $l$-power
index still satisfies all the other properties defining ($M$-)bigness, and one could imagine 
replacing the definition of bigness with one only demanding this weaker property.
If we were to make this change in the definition, the condition that $l\nmid[k:\F_l]$
in the theorem could be removed.
\end{remark}

In applications, one will often be working with crystalline Galois representations which have
regular Hodge-Tate numbers, and one may have the flexibility to choose $l$ large compared to
those Hodge-Tate numbers. (For instance, one might be working with a 
compatible family of Galois representations.)
In this case, we may suppress the third alternative in Theorem \ref{main-thm}.
\begin{lem} \label{regular weights lemma}
Suppose $n$ and $N$ are positive integers, $F$ is a number field and that for each
embedding $\tau:F\into\C$, we have chosen a set 
$a_\tau$ of $n$ distinct integers. Then there exists a number 
$D(n,N,a)$ with the following property. Whenever:
\begin{itemize}
\item  $l>D(n,N,a)$ is a rational prime which is unramified in $F$, 
\item $L\subset\Qbar_l$ is a finite extension of $\Q_l$ (with ring of integers $\cO_L$, residue field $k$),
\item $\rho:G_F\to\GL_n(\cO_L)$ is a crystalline Galois representation, and
\item $\iota:\bar{\Q}_l\isoto\C$ is an isomorphism such for each embedding 
$\sigma:F\into\bar{\Q}_l$ that the multiset of Hodge-Tate numbers of $\rho$
with respect to it is precisely the set $a_{\iota\circ\sigma}$ (and in particular, 
the Hodge-Tate numbers are distinct),
\end{itemize}
then we can never find an integer $m>1$ and some Artin representation 
$\rho':G_F\to\GL_m(\cO_L)$ with image of order at most $N$, and a mod $l$ 
representation $r'':G_F\to\GL_{m'}(k)$, such that  $\bar{\rho}\cong \bar{\rho}'\otimes r''$.
\end{lem}
\begin{cor} Suppose $n$ and $M$ are positive integers, $F$ is a number field and 
that for each embedding $\sigma:G\into\Qbar_l$,
$a_\sigma$ is a set of $n$ distinct integers. Then there exists a number $E(n,M,a)$ 
with the following property. Whenever:
\begin{itemize}
\item  $l>E(n,M,a)$ is a rational prime which is unramified in $F$, 
\item $L\subset\Qbar_l$ is a finite extension of $\Q_l$ (ring of integers $\cO_L$, residue field $k$), 
\item $\rho:G_F\to\GL_n(\cO_L)$ is a crystalline Galois representation, and
\item $\iota:\bar{\Q}_l\isoto\C$ is an isomorphism such for each embedding 
$\sigma:F\into\bar{\Q}_l$ that the multiset of Hodge-Tate numbers of $\rho$
with respect to it is precisely the set $a_{\iota\circ\sigma}$ (and in particular, 
the Hodge-Tate numbers are distinct),
\end{itemize}
and we moreover have that $\bar{\rho}(G_F)$ fails to be an $M$-big subgroup of $\GL_n(k)$, 
then we may conclude that one of the following holds:
\begin{enumerate}
\item $\bar{\rho}$ does not act absolutely irreducibly on $k^n$, or
\item there is a proper subgroup $G'<G_F$ and representation 
$r':G'\to \GL_m(k)$ such that $\bar{\rho}=\Ind_{G'}^{G_F} r'$.
\end{enumerate}
\end{cor}
\noindent (The lemma is proved at the end of \S \ref{ss: main result}; the corollary is then immediate 
given the theorem.)

\medskip
We now say something about the methods used to prove the our main result above. 
There are three main ingredients. The first is, as we have already mentioned, the ideas of \cite{snowden2009bigness} (and their generalizations in \cite{pjw}). 
We are unable to directly apply their results, and instead
will have to look `under the hood' a little, recapitulating some arguments from their 
paper in slightly modified settings. (See \S\ref{sec: sturdy}.) The second is the 
Aschbacher-Dynkin theorem, a classification of the maximal subgroups of $\GL_n(k)$
not containing $\SL_n(k)$ (here $k/\F_l$ is a finite extension). We prove our main theorem
by applying the Aschbacher-Dynkin theorem to the image of the representation in
question, and then breaking into cases according to what kind of maximal subgroup
the image lies in. In most of these cases, we are either done immediately or can reduce 
to a problem concerning a smaller-dimensional representation for which we may 
inductively assume our result holds, and from this deduce the result for the original 
representation. 

The remaining case is where the image of our subgroup in 
$\GL_n(k)/k^\times$ is an almost simple group. In this case, we apply our third ingredient:
results of Larsen which provide good control over which almost simple
groups might arise in such a situation. In particular, they are either drawn from a finite 
set of groups depending only on $n$ and not on $l$ (in this case, we can arrange for $l$ 
to be large enough that such a representation lifts to characteristic zero), or are groups
of Lie type in characteristic $l$ (in which case the image of our representation is essentially
the image of an algebraic representation, and we can finish our argument by appeal
to the ideas of Snowden and Wiles we have already mentioned.)

Of course, there are various wrinkles---most notably, our inductive argument actually 
proves something slightly stronger than the main theorem, so that our inductive hypothesis 
is strong enough when we need to use it. 

\medskip
We now explain the organization of the rest of this paper. In section \ref{sec: big image},
we recall the definitions of `big' and `$M$-big', and also some elementary properties of
bigness proved by Snowden and Wiles (and also their analogues for $M$-bigness, which
are due to White in \cite{pjw}). In section \ref{sec: group theory} we state the various 
group-theoretical tools we will be using, most notably stating the Aschbacher-Dynkin 
theorem in the form we will use it, and recalling the results of Larsen mentioned above.
In section \ref{sec: sturdy} we introduce the notion of a `sturdy' subgroup of $\GL_n(k)$,
for $k/\F_l$ finite, showing that sturdy subgroups are big and have various desirable
properties. The notion of `sturdiness' is very closely related to the notion of `being the
image of an algebraic representation', and our arguments here draw heavily on 
\cite{snowden2009bigness,pjw}.
Finally, in section \ref{sec: main result} we combine these results to prove the main theorem,
and close the section with some (very vague) comments about the degree to which we 
expect representations which satisfy one of the properties (1)--(3) in Theorem \ref{main-thm}
will nonetheless have big image.

{\bf Acknowledgements.} It will be clear to the reader how much this work depends on
the ideas of Andrew Snowden and Andrew Wiles, and it is a pleasure to acknowledge
my debt to them. I also thank them for keeping me updated on the progress 
of revisions to \cite{snowden2009bigness}, and for encouraging me to press ahead in 
proving the results in this paper. Finally, I am grateful to Toby Gee and Andrew Snowden
for providing very helpful comments on an earlier draft of this paper.

\section{Big image} \label{sec: big image}

\subsection{}
We begin by recalling the definition of $M$-big from \cite{BLGHT}
(see Definition 7.2 there).
\begin{defn}
\label{defn: m-big}
  Let $k/\F_l$ be algebraic and $m$ a positive integer. We say that a
  subgroup $H\subset\GL_n(k)$ of $\GL_n(k)$ is $M$-\emph{big} if the
  following conditions are satisfied.
  \begin{itemize}
  \item $H$ has no $l$-power order quotient.
  \item $H^0(H,\mathfrak{sl}_n(k))=(0)$.
  \item $H^1(H,\mathfrak{sl}_n(k))=(0)$.
  \item For all irreducible $k[H]$-submodules $W$ of
    $\mathfrak{gl}_n(k)$ we can find $h\in H$ and $\alpha\in k$ such
    that:
    \begin{itemize}
    \item $\alpha$ is a simple root of the characteristic polynomial
      of $h$, and if $\beta$ is any other root then
      $\alpha^M\ne\beta^M$.
    \item Let $\pi_{h,\alpha}$ (respectively $i_{h,\alpha}$) denote
      the $h$-equivariant projection from $k^n$ to the
      $\alpha$-eigenspace of $h$ (respectively the $h$-equivariant
      injection from the $\alpha$-eigenspace of $h$ to $k^n$). Then
      $\pi_{h,\alpha}\circ W\circ i_{h,\alpha}\ne 0$.
    \end{itemize}
\end{itemize}
\end{defn}
We will use `big' as a synonym for `$1$-big'. This is consistent with the definition of `big' in \cite{cht}.

The following lemmata establish basic properties of bigness and $M$-bigness which 
will be constantly of use to us. They were essentially proved in \cite{snowden2009bigness}.
(The statements in that paper are in the context of bigness, but
the proofs trivially extend to $M$-bigness, as was noted in \cite{pjw}---see section 2 of that paper.)

\begin{prop} \label{normal subgroups and bigness}
Let $k/\F_l$ be algebraic, and let $M$ and $n$ be positive integers. If $G<\GL_n(k)$ has a normal
subgroup $H$, of index prime to $l$, which is $M$-big, then $G$ is $M$-big.
\end{prop}

\begin{prop} \label{bigness and scalars}
Let $k/\F_l$ be algebraic, and let $M$ and $n$ be positive integers. Suppose $G<\GL_n(k)$. 
Then $G$ is $M$-big if and only if $k^\times G$ is.
\end{prop}

\section{Group theory}\label{sec: group theory}

\subsection{} We remind ourselves of a simple form of the Aschbacher-Dynkin theorem,
mostly to fix notation for what follows. The theorem is due independently to 
Aschbacher and Dynkin, and the version we use here is somewhat less powerful
than the full version in the original (independent) arguments of Aschbacker
(see \cite{aschbacher1984maximal}) and Dynkin (see \cite{dynkin1952maximal} for
the original paper in Russian, or \cite{dynkin2000maximal} for an English translation).
See \cite[\S 3.10.3]{wilsonfinitesimple} for the theorem proved in the precise form
we shall use it. Let us first define certain subgroups of the general linear groups over
a finite field.

\begin{defn} \label{def: Aschbacher Dynkin subgps}
Let $l$ be a prime, $k/\F_l$ a finite extension.
Let us describe certain subgroups of $\GL_n(k)$.
\begin{enumerate}
\item For each linear subspace $V\subset k^n$ of dimension $d$, $0<d<n$, we have a
subgroup $G_V$ which sends $V$ onto itself. We have 
$$G_V\cong (l^{[k:\F_l]d(n-d)})\rtimes(\GL_d(k)\times\GL_{n-d}(k))$$
(where $(l^{[k:\F_l]d(n-d)})$ denotes some group of order 
$l^{[k:\F_l]d(n-d)}$). (For obvious reasons, 
we call $G_V$ a \emph{reducible} subgroup.)
\item For each direct sum decomposition $k^n=V_1\oplus\dots\oplus V_m$,
where $m|n$ and each $V_i$ has dimension $n/m$,  we have a subgroup 
$G_{V_1\oplus\dots\oplus V_m}$ which stabilizes the direct sum decomposition
(but need not stabilize the individual summands in the direct sum). We have
$G_{V_1\oplus\dots\oplus V_m}\cong \GL_{n/m}(k)\wr S_m$
and we call $G_{V_1\oplus\dots\oplus V_m}$ an \emph{imprimitive} subgroup.
\item For each decomposition $k^n\equiv V_1\otimes V_2$, where 
$n=\dim V_1\dim V_2$, we have a subgroup $G_{V_1\otimes V_2}$ stabilizing
the tensor product decomposition. We have 
$G_{V_1\otimes V_2}\cong \GL_{n_1}(k)\circ\GL_{n_2}(k)$, where $n_i=\dim V_i$
and $\circ$ denotes the central product.
\item For each decomposition $k^n\equiv V_1\otimes \dots\otimes V_m$, where 
the $V_i$ all have the same dimension ($d$ say), and so $d^m=n$, we have 
a subgroup $G_{V_1\otimes \dots\otimes V_m}$ of 
transformations which preserve this tensor product decomposition 
(but which may rearrange the factors amongst themselves). We have
$$k^\times G_{V_1\otimes \dots\otimes V_m}/k^\times\cong \PGL_d(k)\wr S_m.$$
\item  If $n=p^m$ for some odd prime $p$, then we have certain subgroups 
of $\GL_n(\bar{\Q})$ isomorphic to the `extraspecial' group $p_+^{1+2m}$, whose
normalizers in $\GL_n(\bar{\Q})$ are isomorphic to 
$p_+^{1+2m}\rtimes \Sp_{2m}(\F_p)$. (See \S3.10.2 of \cite{wilsonfinitesimple}.)
Assuming our finite field $k$ is large enough, these subgroups reduce mod $k$
to give a subgroup of $\GL_n(k)$. We can form a larger subgroup, $G_{p^{1+2m}}$,
by additionally including all central elements, and taking subgroup this generates.

If $n=2^m$ for some $m$, we have similar subgroups  $2_\eps^{1+2m}$ 
of $\GL_n(\bar{\Q})$ where $\eps\in\{+,-\}$.  whose normalizers are 
$2_\eps^{1+2m}\rtimes \GO_{2m}^\eps(\F_2)$, and again, if $k$ is large enough,
these subgroups reduce mod $k$
to give a subgroup of $\GL_n(k)$. We can form a larger subgroup, $G_{2_{\eps}^{1+2m}}$,
by additionally including all central elements, and taking subgroup this generates.

We call $G_{p^{1+2m}}$ and $G_{2_{\pm}^{1+2m}}$ \emph{subgroups of extraspecial type}.

\item Let $H$ be an almost simple group (a group $H$ satisfying $G<H<\Aut G$ for some
simple group $G$). Let $\bar{r}:H\to \PGL_n(k)$ be an irreducible projective modular 
representation. Then we can form a subgroup $G_{\bar{r}}:=\bar{r}(H)k^\times $.
\end{enumerate}
\end{defn}

\begin{thm}[Aschbacher, Dynkin] \label{thm: Aschbacher Dynkin thm}
Let $l$ be a prime, $k/\F_l$ a finite extension.
Let $G$ be a subgroup of $\GL_n(k)$ which does not contain $\SL_n(k)$. Then
$G$ is contained in one of the subgroups of the forms listed 
in Definition \ref{def: Aschbacher Dynkin subgps}. 
\end{thm}

\subsection{} We also recall the following result of Larsen, which is essentially a combination
of \cite[Lemmas 1.5--1.8]{larsen1995maximality}.

\begin{prop}[Larsen] \label{Larsen-prop constructing set S}
Fix a semisimple group scheme $\cG/\Z[1/N]$. Then there is
an integer $A$ and a finite collection of finite simple groups $S$, such that for all 
$l>A$, all finite extensions $k/\F_l$, and all subgroups $H<\cG(k)$ which are 
nonabelian finite simple groups, we have:
\begin{enumerate}
\item $H$ is isomorphic to some member of $S$, or
\item $H$ is a derived group of an adjoint group of Lie type, 
$H\cong \cD(\cH(\F_q))$,  $\cH$ simple, which moreover satisfies $l|q$ and
$\cD(\cH(\F_q))=\text{Im}(\cH^{sc}(\F_q)\to\cH(\F_q))$ (where $\cH^{sc}$ denotes
the simply connected cover of $\cH$).
\end{enumerate}
\end{prop}
\begin{proof} If $k=\F_l$, this follows immediately from \cite[Lemmas 1.5--1.8]{larsen1995maximality},
 together with the classification
of finite simple groups (as given, say, in \cite[\S 1.3]{larsen1995maximality}). To see
that the proposition holds in the more general case we give here, first note that 
the proof of \cite[Lemma 1.4]{larsen1995maximality} actually gives a more general
result, where we replace $\F_l$ with any finite extension $k/\F_l$. Then Lemmas 1.5--1.8
have similar generalizations, again with essentially unaltered proofs. The proposition
follows. The fact that we may take $\cD(\cH(\F_q))=\text{Im}(\cH^{sc}(\F_q)\to\cH(\F_q))$
is arranged by including the finitely many finite simple groups of Lie type for which this is 
not true in $S$.
\end{proof}

\subsection{} Finally, we rehearse some of the standard theory which relates, for a
Chevalley group over a finite field of characteristic $l$, modular representations (in characteristic
$l$) of the Chevalley group considered as an algebraic group with algebraic representations
of the Chevalley group; and theory which allows us to lift these algebraic representations to
characteristic zero under certain circumstances.

In particular, we have on the one hand the following result of Steinberg (see 
\cite[1.3]{steinberg1963representations} and \cite[13.3]{steinberg1968endomorphisms}):
\begin{thm} \label{steinberg-rep-thm}
Suppose that $k/\F_l$ is a finite extension, and $\cG/k$ is an almost simple,
simply connected algebraic group for which we have made a choice of a maximal torus
and a system of simple roots. Then there is a bijection between irreducible 
projective $\bar{\F}_l$
representations of the abstract group $\cG(k)$ and tensor products:
$$\bigotimes_{i=0}^{[k:\F_l]-1} V_i^{\Frob^i}$$
where each $V_i$ is an irreducible algebraic representation with highest weight $\lambda$ satisfying
$0\leq \lambda(\mathbf{a}) \leq l-1$ for each simple root $\mathbf{a}$, 
and where the superscript `$Frob^i$' indicates precomposing
a representation (considered as a map $\cG(k)\to\PGL_n(\bar{\F}_l)$ for some $n$) with the $i$th
power of the Frobenius map $\cG(k)\to\cG(k)$.
The bijection is given by mapping such tensor products
into abstract group representations by restricting them to $k$ points in the obvious way.
\end{thm}

And, on the other, the following result which is essentially due to Larsen:
\begin{thm} \label{larsen-ss-thm}
Suppose that we are in the situation of the previous theorem, and that $\cG$
is split over $k$ and $l\geq 3(\dim V_i)^2$. Then each of 
the representations $V_i$ lifts to characteristic zero, in the sense that there is some algebraic
group $\mathfrak{G}_i$ over $W(k)$ and algebraic representation $\tilde{V}_i$ of 
$\mathfrak{G}_i$, such that base changing back to $k$ we recover $\cG$ and $V_i$.
\end{thm}
\begin{proof} The proof of Proposition 4.4 of \cite{larsen1995semisimplicity} extends to give us what
we need. See also \cite{jantzen1997low}.
\end{proof}

We will record a trivial corollary of this second theorem, which will be very important to us.
\begin{corollary} \label{norm-bounding-corollary}
Suppose $n$ is a positive integer. Then there are constants $C_0(n)$ and
$C_1(n)$ with the following property. 
Suppose that $l>C_0(n)$ is a prime, 
that $k/\F_l$ is a finite extension, that $\cG/k$ is an almost simple,
simply connected algebraic group, 
and that we are given an irreducible 
$\bar{\F}_l$ representation of the abstract group $\cG(k)$, at most $n$ dimensional, which corresponds to 
$$\bigotimes_{i=0}^{[k:\F_l]-1} V_i^{\Frob^i}$$
under the correspondence of Theorem \ref{steinberg-rep-thm}.
Then the norms $||V_i||$ of the representations $V_i$ (see \S3.2 of \cite{snowden2009bigness})
satisfy $||V_i||<C_1(n)$.
\end{corollary}
\begin{proof} This follows straightforwardly from the fact that the $V_i$ lift to characteristic 0; see 
Proposition 3.5 of \cite{snowden2009bigness}.
\end{proof}

\section{Sturdy subgroups}\label{sec: sturdy}
\subsection{}
In this section, we will define the notion of a \emph{sturdy} subgroup of $\GL_n(k)$; the
point of introducing this notion is that on the one hand, being sturdy suffices to be $M$-big, 
at least for $l$ large; and on the other hand, being sturdy behaves well under `taking tensor 
products' (see Lemma \ref{tensor product of sturdy subgroups sturdy}). Sturdy subgroups
are very closely related to the images of algebraic representations, and as such, this 
section will draw heavily on the work of Snowden and Wiles discussed above. One
wrinkle, however, is that we will need to deal with representations which are tensor 
products of algebraic representations defined over different fields.

\begin{defn} \label{definition of very sturdy} Suppose $l$ is a prime number, $n$ an integer, $k/\F_l$ 
a finite extension, and $\Gamma<\GL_n(k)$ a subgroup. We say that $\Gamma$ is 
\emph{very sturdy} if it contains $k^\times$, acts absolutely irreducibly, and we can find:
\begin{itemize}
\item a sequence $k_1,\dots,k_m$ of fields, each of which is an intermediate field between $k$ and $\F_l$,
\item for each $i$, $1=1,\dots,m$, an almost simple, connected,
simply connected algebraic group $\cG^{sc}_i/k_i$, 
which remains almost simple when we base change to $\bar{k}$, and
\item for each $i$, $1=1,\dots,m$, a faithful projective representation over $k$ of the abstract group $\pi_i(\cG_i^{sc}(k_i))$ 
$$r_i:\pi_i(\cG_i^{sc}(k_i))\to\PGL_{n_i}(k),$$
where
\begin{itemize}
\item $\cG_i/k_i$ is the algebraic group we get by taking the quotient of $\cG^{sc}_i/k_i$ by its center (which will be its unique maximal normal subgroup), and
\item $\pi_i:\cG_i^{sc}\to \cG_i$ is the covering map, which we also consider as a map of $k_i$ points,
\end{itemize}
(so that $\pi_i(\cG^{sc}(k_i))$ is a finite simple group of Lie type)
\end{itemize}
such that:
\begin{itemize}
\item the abstract group $\pi_i(\cG^{sc}(k_i))$ is in fact a simple group for all $i$,
\item $n=\prod_i n_i$ (this will follow automatically from the other conditions, but we mention
it to orient the reader), and
\item $\Gamma/k^\times$ is conjugate to the subgroup of $\PGL_n(k)$ given by the image of the map
$$\prod_i \pi_i(\cG_i^{sc}(k_i)) \overset{\prod_i r_i}{\longrightarrow} \prod_i \PGL_{n_i}(k) \into \PGL_{n}(k).$$
\end{itemize}
\end{defn}

\begin{defn} \label{definition of sturdy} Suppose $l$ is a prime number, $k/\F_l$ 
a finite extension, and $\Gamma<\GL_n(k)$ a subgroup.
We say that $\Gamma$ is \emph{sturdy} if there is a chain $\Gamma_1\triangleleft 
\Gamma_2 \triangleleft\dots\triangleleft \Gamma_r=\Gamma$ of subgroups of 
$\Gamma$, each normal in the next, such that:
\begin{enumerate}
\item for each $i$, $1\leq i \leq r-1$, $(\Gamma_{i+1}:\Gamma_{i})$ is coprime to $l$, and
\item $\Gamma_1$ is very sturdy, in the sense of Definition \ref{definition of very sturdy}.
\end{enumerate}
\end{defn}
\noindent Here are an easy observation and a remark:
\begin{lem} \label{sturdy subgps up to iso}
Suppose $l$ is a prime number, $n$ an integer, $k/\F_l$ 
a finite extension, and $\Gamma<\GL_n(k)$ a very sturdy subgroup. 
Then, as an abstract group $\Gamma/k^\times$ is a product of the
simple groups $\pi_i(\cG^{sc}(k_i))$, which are simple groups of Lie
type in characteristic $l$.
\end{lem}

\begin{rem} \label{nicely presented very sturdy subgroups}
If $\Gamma$ is a very sturdy subgroup, then we have the freedom to replace 
any of the groups $\cG^{sc}_i$ referred to in Definition \ref{definition of very sturdy} with
other groups which yield an isomorphic $\pi_i(\cG_i^{sc}(k_i))$ (and then modify the 
corresponding $r_i$ by composing with the isomorphism between the old and new 
$\pi_i(\cG_i^{sc}(k_i))$s), and the new $\cG_i^{sc}$s will still satisfy the properties 
required to show that $\Gamma$ is very sturdy. As explained in \cite[\S2.4--1,\S3.2]{conway-atlas},
for any finite simple group of Lie type in characteristic $l$ (that is, any group which could arise
amongst the $\pi_i(\cG_i^{sc}(k_i))$s above),
it is possible to find a $\cG^{sc}$ which gives rise to it as above such that 
the algebraic group $\cG_i^{sc}$ in addition 
enjoys the following property: it
has a Borel subgroup $\cB$ and torus $\cT$, such that the Borel and torus are 
defined over $k$ and hence
sent onto themselves by $\Gal(\bar{k}/k)$. A generator of $\Gal(\bar{k}/k)$
acts as a graph automorphism of the
Dynkin diagram of $\cG^{sc}$, and so the simple roots can be arranged in cycles under
this action. We will from time to time assume that the $\cG_i^{sc}$ have been 
chosen with these special properties. In this case, we will write $\cT_i$ for the
torus in $\cG_i$ according to the properties just described.
Thus, in particular the finite group $\cT(k_i)$  is a product of $\Gm$s of extension
fields of $k_i$, one for each orbit of the simple roots under Frobenius, where for
each orbit the extension field has degree over $k_i$ equal to the number of roots
in the orbit.
\end{rem}

\subsection{Properties of sturdy subgroups}

\begin{lemma} \label{tensor product of sturdy subgroups sturdy}
Suppose that $n_1$, $n_2$ are positive integers, and we write $n=n_1n_2$. Suppose we have 
chosen an isomorphism $k^n\cong k^{n_1}\otimes k^{n_2}$; this then gives us a  natural map 
$\PGL_{n_1}(k)\times\PGL_{n_2}(k)\into \PGL_n(k)$, whose image is the collection of elements
of $\PGL_n(k)$ preserving the tensor product decomposition. Suppose that $\Gamma<\PGL_n(k)$
is a subgroup, which is contained in the image of $\PGL_{n_1}(k)\times\PGL_{n_2}(k)\into 
\PGL_n(k)$, so we can think of $\Gamma$ as a subgroup of $\PGL_{n_1}(k)\times\PGL_{n_2}(k)$.
Let $\pi_1:\PGL_{n_1}(k)\times\PGL_{n_2}(k)\to\PGL_{n_1}(k)$ denote the projection map, and
similarly for $\pi_2$. Let $q:\GL_{n}(k)\onto\PGL_{n}(k)$ denote the natural quotient 
map, and similarly $q_1$ and $q_2$. Suppose that $q_1^{-1}(\pi_1(\Gamma))$ and 
$q_2^{-1}(\pi_2(\Gamma))$ are both sturdy. Then $q^{-1}(\Gamma)$ is sturdy
if it acts absolutely irreducibly.
\end{lemma}
\begin{proof} We shall write $\Gamma^{(1)}$ for $\pi_1(\Gamma)$ and
$\Gamma^{(2)}$ for $\pi_2(\Gamma)$. Goursat's Lemma tells us that we can 
find a group $G$ and surjections $p^{(1)}:\Gamma^{(1)}\to G$, $p^{(2)}:\Gamma^{(2)}\to G$,
such that $\Gamma$, considered as a subgroup of $\PGL_{n_1}(k)\times\PGL_{n_2}(k)$,
is precisely the set
$$\{(\gamma^{(1)},\gamma^{(2)}) \in \Gamma^{(1)}\times\Gamma^{(2)} \,|\,  
p^{(1)}(\gamma^{(1)})=p^{(2)}(\gamma^{(2)})\}=\Gamma^{(1)}\times_G\Gamma^{(2)}.$$
We will write $N^{(1)}$ for $\ker p^{(1)}$, and $N^{(2)}$ for $\ker p^{(2)}$. We will write $p$
for the obvious homomorphism $\Gamma\to G$ sending $(\gamma^{(1)},\gamma^{(2)})$ to the common
value of $p^{(1)}(\gamma^{(1)})$ and $p^{(2)}(\gamma^{(2)})$.

We are given that $q_1^{-1}(\Gamma^{(1)})$ and $q_2^{-1}(\Gamma^{(2)})$ are sturdy, from which we 
deduce the existence of sequences of subgroups 
$$\Gamma^{(1)}_1\triangleleft  \Gamma^{(1)}_2 \triangleleft\dots\triangleleft  \Gamma^{(1)}_r=\Gamma^{(1)},
\quad 
\Gamma^{(2)}_1\triangleleft  \Gamma^{(2)}_2 \triangleleft\dots\triangleleft  \Gamma^{(2)}_s=\Gamma^{(2)}$$
such that for each sequence each group is normal in the next, with $q_i^{-1}(\Gamma_1^{(i)})$ very sturdy
and $(\Gamma^{(i)}_{j+1}: \Gamma^{(i)}_{j})$ coprime to $l$ for each $i,j$. In particular 
$\Gamma_1^{(i)}$ and $\Gamma_2^{(i)}$ are each abstractly isomorphic to a product of
finite simple groups, each of which is a simple group $\cD(\cH(\F_q))$ of Lie type with $q|l$.
(See Lemma \ref{sturdy subgps up to iso}.)
For brevity in the remainder of the proof, we will refer to such a product as a $l$-Lie-ish group. 
We remark that any quotient of an $l$-Lie-ish group is again an $l$-Lie-ish group, since 
any normal subgroup of a product of simple perfect groups is a product of some subset of the factors.

From these, we get sequences of subgroups 
$$G^{(1)}_1\triangleleft G^{(1)}_2\triangleleft\dots\triangleleft G^{(1)}_r=G,\quad 
G^{(2)}_1\triangleleft G^{(2)}_2\triangleleft\dots\triangleleft G^{(2)}_s=G$$
of $G$ (such that again for each sequence each group is normal in the next) by putting 
$G^{(i)}_j = p^{(i)}(\Gamma_j^{(i)})\cong N^{(i)} \Gamma_j^{(i)}/N^{(i)}$. 
Since $G_1^{(1)}$ is a quotient of $\Gamma_1^{(i)}$, which is $l$-Lie-ish,
$G_1^{(1)}$ is also $l$-Lie-ish; and similarly for $G_1^{(2)}$.

\begin{claim} We have that $G_1^{(1)}\cap G_1^{(2)}=G_1^{(1)}$, so $G_1^{(1)}\subset G_1^{(2)}$.
\end{claim}
\begin{proof} We shall show, by downward induction on $j$ from $s$ to 1, that 
$G_1^{(1)}\cap G_j^{(2)}=G_1^{(1)}$. The base case is trivial. Assuming that 
$G_1^{(1)}\cap G_j^{(2)}=G_1^{(1)}$, we want to understand $G_1^{(1)}\cap G_{j-1}^{(2)}$.
We know $G_1^{(1)}\cap G_{j-1}^{(2)}\triangleleft G_1^{(1)}\cap G_j^{(2)}=G_1^{(1)}$. 
If the inclusion were proper, then $G_1^{(1)}/G_1^{(1)}\cap G_{j-1}^{(2)}$ 
would be a nontrivial quotient of $G_1^{(1)}$, and hence (since $G_1^{(1)}$ is 
$l$-Lie-ish) would be $l$-Lie-ish itself, so would
have order divisible by $l$. But this contradicts the fact that 
$G_{j-1}^{(2)}: G_j^{(2)}$
and hence
$G_1^{(1)}\cap G_{j-1}^{(2)}: G_1^{(1)}\cap G_j^{(2)}$
are coprime to $l$.
\end{proof}
\noindent
The same argument gives that $G_1^{(2)}\subset G_1^{(1)}$, whence we see that $G_1^{(1)}= G_1^{(2)}$.
We'll adopt $G_1$ as a shorter name for the common value.

\medskip
For each $G^{(1)}_j$, we can construct a group $\Gamma_j=p^{-1}(G^{(1)}_j)< \Gamma$; we see
that each $\Gamma_j$ is a normal subgroup of the next with prime-to-$l$ index. Thus it will
suffice for us to show $q^{-1}(\Gamma_1)$ is sturdy to deduce that $q^{-1}(\Gamma)$ is sturdy.
We put $\Gamma'=\Gamma_1$. Note that
\begin{align*}
\Gamma'&=\{(\gamma^{(1)},\gamma^{(2)})\in \Gamma_1^{(1)}N^{(1)}\times\Gamma_1^{(2)}N^{(2)}
\,|\,  
p^{(1)}(\gamma^{(1)})=p^{(2)}(\gamma^{(2)})\}\\
&\cong\Gamma_1^{(1)}N^{(1)}\times_{G_1}\Gamma_1^{(2)}N^{(2)} 
\end{align*}

For each $j$, $1\leq j \leq r$, we can form a subgroup $N^{(1)}\Gamma^{(1)}_1 \cap \Gamma^{(1)}_j$
of $N^{(1)}\Gamma^{(1)}_1$; we see that each is normal in the next, the successive quotients have
order prime to $l$, and for each of these groups, 
$p^{(1)}(N^{(1)}\Gamma^{(1)}_1 \cap \Gamma^{(1)}_j)=G_1$. We then form a subgroup
$\Gamma'_j$ of $\Gamma'$, by putting
\begin{align*}
\Gamma'_j&=\{(\gamma^{(1)},\gamma^{(2)})\in (\Gamma_1^{(1)}N^{(1)}\cap \Gamma^{(1)}_j)\times\Gamma_1^{(2)}N^{(2)}
\,|\,  
p^{(1)}(\gamma^{(1)})=p^{(2)}(\gamma^{(2)})\}\\
&\cong((\Gamma_1^{(1)}N^{(1)})\cap \Gamma^{(1)}_j)\times_{G_1}\Gamma_1^{(2)}N^{(2)} 
\end{align*}
We see that each $\Gamma'_j$ is normal in the next, and the successive quotients have
order prime to $l$; thus it will
suffice for us to show $q^{-1}(\Gamma'_1)$ is sturdy to deduce that $q^{-1}(\Gamma')$
(and hence $q^{-1}(\Gamma)$) is sturdy. We put $\Gamma''=\Gamma'_1$, so
\begin{align*}
\Gamma''&=\{(\gamma^{(1)},\gamma^{(2)})\in \Gamma^{(1)}_1\times\Gamma_1^{(2)}N^{(2)}
\,|\,  
p^{(1)}(\gamma^{(1)})=p^{(2)}(\gamma^{(2)})\}\\
&\cong \Gamma^{(1)}_j\times_{G_1}\Gamma_1^{(2)}N^{(2)} 
\end{align*}

Similarly for each $j$, $1\leq j \leq s$, we can form a subgroup $N^{(2)}\Gamma^{(2)}_1 \cap \Gamma^{(2)}_j$
of $N^{(2)}\Gamma^{(2)}_1$; we see that each is normal in the next, the successive quotients have
order prime to $l$, and for each of these groups, 
$p^{(2)}(N^{(2)}\Gamma^{(2)}_1 \cap \Gamma^{(2)}_j)=G_1$. We then form a subgroup
$\Gamma''_j$ of $\Gamma''$, by putting
\begin{align*}
\Gamma''_j&=\{(\gamma^{(1)},\gamma^{(2)})\in \Gamma^{(1)}_1\times
           (\Gamma_1^{(2)}N^{(2)}\cap \Gamma^{(2)}_j)
\,|\,  
p^{(1)}(\gamma^{(1)})=p^{(2)}(\gamma^{(2)})\}\\
&\cong \Gamma^{(1)}_1\times_{G_1}((\Gamma_1^{(2)}N^{(2)})\cap \Gamma^{(2)}_j))
\end{align*}
We see that each $\Gamma''_j$ is normal in the next, and the successive quotients have
order prime to $l$; thus it will suffice for us to show $q^{-1}(\Gamma''_1)$ is sturdy to 
deduce that $q^{-1}(\Gamma'')$
(and hence $q^{-1}(\Gamma)$) is sturdy. We put $\Gamma'''=\Gamma''_1$; then:
\begin{align*}
\Gamma'''&=\{(\gamma^{(1)},\gamma^{(2)})\in \Gamma^{(1)}_1\times \Gamma^{(2)}_1
\,|\,  
p^{(1)}(\gamma^{(1)})=p^{(2)}(\gamma^{(2)})\}\\
&\cong \Gamma^{(1)}_1\times_{G_1}\Gamma^{(2)}_1
\end{align*}
We know that the $\Gamma^{(i)}_1$ are $l$-Lie-ish groups, so normal
subgroup is the product of a subset of the factors, and thus, writing
$N_1^{(i)}$ for $\Gamma^{(i)}_1\cap N^{(i)}$, so 
$\Gamma^{(1)}_1/ N^{(1)}\cong\Gamma^{(2)}_1/ N^{(2)}\cong G_1$,
we see that this isomorphism in fact identifies a subset of the factors of 
$\Gamma^{(1)}_1$ with a subset of the factors of $\Gamma^{(1)}_2$.

Thus (reordering the factors if necessary), we can write
\begin{align*}
\Gamma^{(1)}_1&\cong\prod_{i=1}^{t}\pi_i(\cG_i^{sc}(k_i)) \times \prod_{i=t+1}^{m^{(1)}}\pi^{(1)}_i(\cG_i^{(1),sc}(k_i))\\
\Gamma^{(2)}_1&\cong\prod_{i=1}^{t}\pi_i(\cG_i^{sc}(k_i)) \times \prod_{i=t+1}^{m^{(2)}}\pi^{(2)}_i(\cG_i^{(2),,sc}(k_i))
\end{align*}
(So the map to $G_1$ is, in either case, projection onto the first $t$ factors.) 

Now $\Gamma^{(1)}_1$ (as a subgroup, rather than up to isomorphism) is the image of 
$$
\prod^{t}_{i=1} \pi_i(\cG_i^{sc}(k_i))\times
\prod^{m^{(1)}}_{i=t+1} \pi^{(1)}_i(\cG_i^{(1),sc}(k_i)) \overset{\prod_i r^{(1)}_i}{\longrightarrow} \prod_i \PGL_{n^{(1)}_i}(k) \into \PGL_{n_1}(k)$$
for some projective representations $r^{(1)}_i$, $i=1,\dots ,m^{(1)}$, of dimension $n^{(1)}_i$; 
and similarly for $\Gamma^{(2)}_1$ (with projective 
representations $r^{(2)}_i$ of dimension $n^{(2)}_i$). 
Thus for each $i$, $i\leq t$, $r^{(1)}_i$ and $r^{(2)}_i$ are projective 
representations of a common group $\pi_i(\cG_i^{sc}(k_i)$, and we can form their tensor product, 
$r'_i$ say.

Then we see $\Gamma'''$ is (conjugate to) the image of
$$\xymatrix{
\prod^{t}_{i=1} \pi_i(\cG_i^{sc}(k_i)) \times 
\prod^{m^{(1)}}_{i=t+1} \pi^{(1)}_i(\cG_i^{(1),sc}(k_i)) \times 
\prod^{m^{(2)}}_{i=t+1} \pi^{(2)}_i(\cG_i^{(2),sc}(k_i))
\ar[d]^{\prod_{i=1}^t r'_i \times \prod_{i=t+1}^{m^{(1)}} r^{(1)}_i \times \prod_{i=t+1}^{m^{(2)}} r^{(2)}_i}
\\
\prod_{i=1}^t \PGL_{n^{(1)}_i n^{(2)}_i}(k) \times \prod_{i=t+1}^{m^{(1)}} \PGL_{n^{(1)}_i}(k) \times \prod_{i=t+1}^{m^{(2)}} \PGL_{n^{(2)}_i}(k)
\ar@{^{(}->}[d]
\\ \PGL_{n}(k).
}$$
This demonstrates that $\Gamma'''$ is very sturdy and hence sturdy, so $\Gamma$ is sturdy.
\end{proof}

Using a few of the same ideas, we can also prove:
\begin{lem} \label{normal subgp of sturdy of prime to l index is sturdy}
Suppose $l$ is a prime, $n$ is a positive integer, $k/\F_l$ is a finite extension, 
$\Gamma$ is a sturdy subgroup of $\GL_n(k)$, and $N$ is a normal subgroup 
of index prime to $l$. Then $N$ is also sturdy.
\end{lem}
\begin{proof}
By assumption, we have a sequence $\Gamma_1\triangleleft \Gamma_2 \triangleleft
\dots\triangleleft \Gamma_r=\Gamma$ of subgroups, each normal in the next with
index prime to $l$, and with $\Gamma_1$ very sturdy. Then we have a sequence
$\Gamma_1\cap N\triangleleft \Gamma_2 \cap N\triangleleft\dots
\triangleleft \Gamma_r\cap N=N$ of subgroups of $N$, each normal in the next with
index prime to $l$; so if we show $\Gamma_1\cap N$ very sturdy we will be done.
$\Gamma_1\cap N:\Gamma_1$ is prime to $l$, whereas $\Gamma_1$ is $l$-Lie-ish
(as discussed in the proof above), so every proper quotient is $l$-Lie-ish and
hence has order dividing $l$. So $\Gamma_1\cap N=\Gamma_1$ and the lemma
follows.
\end{proof}

\subsection{Very regular elements}

In the next subsection, we will prove a result showing that sturdy subgroups are big for $l$
large. But before we turn to this task,we first need a lemma allowing the construction of 
`very regular elements', as in \cite[\S4]{snowden2009bigness} and \cite[\S3]{pjw}. Our 
arguments will have to be a little more complicated, however, since we will have a 
collection of different algebraic groups (defined over different fields), and we need to
elements in each which are `very regular in combination'.

The arguments in this section are rather technical but are elementary and involve 
little of enduring interest. We might advise the reader to skip them at a first reading,
studying only the statement of Lemma \ref{very regular elements lemma} before moving
on to the next subsection. (We remark that had we been willing to introduce a bound on the degree
$[k:\F_l]$ in the conditions on our main theorems, and let the bounds on $l$ in the main theorems 
depend on the bound on this degree (which is probably harmless in most applications), we
could have given radically simpler proofs of our results here, essentially appealing to results
of \cite{snowden2009bigness} applied to appropriate restrictions of scalars.)

We will establish a succession of stronger and stronger lemmas building up to the main
result of the subsection.

\begin{lem}\label{lem431}
Suppose $\Omega$, $\Xi$, $N$ are positive integers. There is an integer $L$ with the following 
property.

Suppose that $d$ is a positive integer and 
$\vec{\mu}=(\mu_0,\dots,\mu_{d-1})\in(\Z\cap[-N,N])^d$, so each $\mu_i$ 
is an integer between $-N$ and $N$. 
Suppose also that at most $\Xi$ of the $\mu_i$ are nonzero, and that $q=p^k$ is a prime power,
with $p^k|d$, such that 
$$\sum_{i=0,\dots,d-1} \zeta_d^{p^k i} \mu_i \neq 0$$

Now suppose finally that $l>L$ is a prime number, and let $h=(\sum_j \mu_j l^j)$. 
Then $\#\{ht|t\in \Z/(l^d-1)\Z\} > (\sqrt[q]{\Omega})^{d/\log\log d}.$
\end{lem}
\begin{proof} 
We begin by choosing the number $L$. We first recall that there is a constant $K$
such that $\phi(n)>Kn/\log\log n$, where $\phi$ is Euler's $\phi$ function---see 
\cite[Theorem 328]{hardywright}. We then take $L$ to be large enough that
$L>2\Xi\sqrt[K]{\Omega} N$.

We now move on to the proof proper.
We first consider the linear map $\psi:\Q^d\to\Q^d$ given my `convolution by $\mu$'
---that is, the map
$$(t_0,\dots,t_{d-1})\mapsto(\sum_{i=0}^d t_i \mu_{-i},\sum_{i=0}^d t_i \mu_{1-i},\dots, \sum_{i=0}^d t_i \mu_{d-1-i})$$
---and we claim it has image which is at least $\phi(d/p^k)$ dimensional. 
To see this, we note that as `convolution is Fourier dual to pointwise multiplication',
it suffices to show that, if we write $(\hat{\mu}_0,\dots,\hat{\mu}_{d-1})$ for the Fourier transform of
$\mu$, so $\hat{\mu}_i = \sum_j \zeta_d^{ij} \mu_j$, then the linear map $\Q^d\to\Q^d$ defined by
$$(\hat{t}_0,\dots,\hat{t}_{d-1})\mapsto (\hat{t}_0\hat{\mu}_0,\hat{t}_1\hat{\mu}_1,\dots,\hat{t}_{d-1}\hat{\mu}_{d-1})$$
has image which is at least $\phi(d/p^k)$ dimensional. In other words, we must show that at
least $\phi(d/p^k)$ of the $\hat{\mu}_j$ are nonzero. But, by hypothesis, we know that 
$\sum_{i=0,\dots,d-1} \zeta_d^{p^k i} \mu_i \neq 0$; and applying elements of $\Gal(\Q(\zeta_d)/\Q)$
we see that $\sum_{i=0}^{d-1} \zeta_d^{p^k \eta i} \mu_i \neq 0$ for all $\eta$ coprime to $d$.
This tells us that $\hat{\mu}_j$ is nonzero whenever $p^k|j$ and $j/p^k$ is coprime to $d/p^k$. 
There are $\phi(d/p^k)$ such numbers. This proves the claim in the first sentence of the paragraph.

We note that we can consider $\psi$ also as a map $\Z^d\to\Z^d$. Then the result of the 
previous paragraph tells us that the cardinality of the image $\psi(([0,\sqrt[K]{\Omega}]\cap \Z)^d)$
satisfies:
\begin{align*}
\#\psi(([0,\sqrt[K]{\Omega}]\cap \Z)^d) &> (\sqrt[K]{\Omega})^{\phi(d/p^k)} \\
\intertext{and we then deduce moreover that this}
&> (\sqrt[K]{\Omega})^{\frac{K d}{p^k \log(d/p^k)}}
> (\sqrt[K]{\Omega})^{\frac{K d}{p^k \log(d)}}
= (\sqrt[q]{\Omega})^{\frac{d}{\log(d)}}.
\end{align*} 
On the other hand, examining the image $\psi(([0,\sqrt[K]{\Omega}]\cap \Z)^d)$, we see that 
for each $(\nu_0,\dots,\nu_{d-1})\in\psi(([0,\sqrt[K]{\Omega}]\cap \Z)^d)$, each component
$\nu_i$ is a sum $\sum_{j=0}^d t_j \mu_{i-j}$ with at most $\Xi$ nonzero terms (since at most 
$\Xi$ of the $\mu_i$ are nonzero), each term of which is a product of an integer $<\sqrt[K]{\Omega}$
and an integer between $-N$ and $N$. Thus 
$$\nu_i\in[-\Xi N\sqrt[K]{\Omega},\Xi N\sqrt[K]{\Omega}]$$
for each $i$ and $\nu$. Thus $\psi(([0,\sqrt[K]{\Omega}]\cap \Z)^d)$ lies in the set
$$S:=([-\Xi N\sqrt[K]{\Omega},\Xi N\sqrt[K]{\Omega}]\cap \Z)^d.$$

But we observe that the map $\alpha:S\to \Z/(l^d-1)\Z$ mapping $(\nu_0,\dots,\nu_{d-1})\mapsto
\sum \nu_i l^i$ is injective. (This depends on the fact that $l>L>2\Xi\sqrt[K]{\Omega} N$.)
Thus $\#\alpha(\psi(
([0,\sqrt[K]{\Omega}]\cap \Z)^d
))>(\sqrt[q]{\Omega})^{d/\log(d)}$.

On the other hand, for $(t_0,\dots,t_{d-1})\in \Z^d$, we have
that 
\begin{align*}
\alpha(\psi(t_0,\dots,t_{d-1})) &= \sum_i l^i \sum_j t_j \mu_{i-j} \text{ (mod $l^d-1$)}\\
&=\sum_j\sum_m t_j l^{j+m} \mu_m \\
\intertext{(putting $i=j+m$ and changing the order of summation; notice that not
only do we take the subscripts mod $d$, but we can also take the power of $l$ 
mod $d$, since we work overall mod $l^d-1$)}
&=h \sum_j t_j l^j \in \{ht|t\in \Z/(l^d-1)\Z\}
\end{align*}
whence $\#\{ht|t\in \Z/(l^d-1)\Z\}>\#\alpha(\psi(
([0,\sqrt[K]{\Omega}]\cap \Z)^d
))>(\sqrt[q]{\Omega})^{d/\log(d)}$, as required.
\end{proof}

\begin{lem} \label{lem432}
Suppose that $\Xi$ is a positive integer. 
There is an integer $n$ with the following property. 

Suppose that $d$ is a positive integer, and $\vec{\mu}=(\mu_0,\dots,\mu_{d-1})\in \Z^d$.
Suppose that both of the following conditions hold:
\begin{itemize}
\item We have $\#\{i|\mu_i\neq 0\} < \Xi$.
\item Whenever $p^k$ is a prime power, with $p^k<n$ and $p^k|d$, we have that
$$\sum_{i=0,\dots,d-1} \zeta_d^{p^k i} \mu_i = 0$$
where $\zeta_d$ denotes a primitive $d$th root of 1.
\end{itemize}
Then $\vec{\mu}$ is identically 0.
\end{lem}
\begin{proof} We take $n=\Xi+1$.
Let us begin by establishing some notation. Let $V_1$ denote the
vector space of $\Q$-valued functions on the group $\Z/d\Z$. We think of $\vec{\mu}$
as an element of $V_1$. For each positive integer
$q$ dividing $d$, let $V_d$ denote the subspace of functions $f$ such that 
$f(a+(d/q))=f(a)$ for all $a\in\Z/d\Z$. For each
$d'|d$, let $W^{(d')}_1$ denote the vector space of $\Q$-valued functions on $\Z/d'\Z$,
for each prime power $q|d'$, let $W^{(d')}_q$ denote the set $\{f\in W^{(d')}|
f(a+(d'/q))=f(a) \text{ for all }a\in\Z/d\Z\}$. Thus $W^{(d)}_q$ is an alias for
$V_q$. Finally, we note that we have a map $\pi_{d'} : V_1\to W^{(d')}_1$ mapping
a function $f:\Z/d\Z\to\Q$ to the function $g$ sending $a\in\Z/d'\Z$ to the sum
of $f$ over elements $b\in\Z/d\Z$ which give $a$ when reduced mod $d'$. There
is a similar map $\pi_{d',d''}:W^{(d')}_1\to W^{(d'')}_1$ for all $d''|d'|d$.

\begin{claim} Suppose that $\sum_{i=0}^{d-1} \zeta^{i}_d \mu_i=0$. Then 
$\vec{\mu}\in \sum_{p\text{ a prime},p|d} V_p$, where the sigma denotes
the internal sum of subspaces of $V_1$.
\end{claim}
\begin{proof} We consider $\Q(\zeta_d)$ as a $\Q$ vector space (which is
$\phi(d)$ dimensional, by the standard theory of cyclotomic extensions),
and consider the linear map of $\Q$ vector spaces $\alpha:V_1\to \Q(\zeta_d)$
mapping $f\mapsto \sum_{i=0}^{d-1} f(i)\zeta_d^i$, which is clearly surjective.
Hence $\ker\alpha$ is $d-\phi(d)$ dimensional. On the other hand, for each
prime $p$ dividing $d$, we readily see that $V_p\subset \ker\alpha$, so that
$\sum_{p\in P}V_p \subset \ker \alpha$, where $P$ denotes the set of primes
dividing $d$. On the other hand, 
\begin{align*}
\dim\sum_{p|d,\text{ $p$ prime}}V_p
&=\sum_{S\subset P,S\neq\emptyset} (-1)^{\#S} \dim\left(\bigcap_{p\in S} V_p\right)\\
&=\sum_{S\subset P,S\neq\emptyset} (-1)^{\#S} \frac{d}{\prod_{p\in S} p} = d-\phi(d).
\end{align*}
It follows that $\ker\alpha=\dim\sum_{p|d,\text{ $p$ prime}}V_p$, and hence, as
we assume $\vec{\mu}\in\ker\alpha$, our conclusion follows.
\end{proof}
The same ideas readily extend to prove:
\begin{claim} Suppose that $q|d$ is a prime power and 
$\sum_{i=0}^{d-1} \zeta^{qi}_d \mu_i=0$. Then 
$$\vec{\mu}\in\pi_{d/q}^{-1}\left(\sum_{p\text{ a prime},p|(d/q)} W^{(d/q)}_p\right).$$
\end{claim}

Now, suppose that $q_1,\dots,q_r$ are all prime powers dividing $d$,
and write $q_1'$ for the next-highest power of the prime dividing $q_1$
after $q_1$ (so if $q_1=p^k$, $q_1'=p^{k+1}$). Then we see that:
\begin{multline}\label{reduction eqn}
\left(\sum_{i=1,\dots,r} V_{d/q_i}\right)\cap\,
\left(\pi_{d/q_1}\left(\sum_{p\text{ a prime},p|(d/q_1)} W^{(d/q_1)}_p\right)\right)
\\
=\begin{cases}
V_{d/q_1'}+\sum_{i=2,\dots,r} V_{d/q_i} &\text{(if $q'_1|d$)}\\
\sum_{i=2,\dots,r} V_{d/q_i} &\text{(if $q'_1\nmid d$)}
\end{cases}
\end{multline}

We are now in a position to proceed to the proof proper. Suppose that 
$\vec{\mu}$ satisfies the two bullet points in the statement of the proposition.
Let $p_1,\dots,p_r$ be the primes dividing $d$.
We know from the first claim above that $\vec{\mu}\in V_{p_1} + \dots V_{p_r}$.
We will think of this as saying $\vec{\mu}\in V_{p_1^{a_1}} + \dots V_{p^{a_r}_r}$,
where each $a_i$ stands for 1 for the time being. Our aim will be to gradually
increase the numbers $a_i$ and/or reduce the number of terms in the sum
$V_{p_1^{a_1}} + \dots V_{p_s^{a_s}}$, always maintaining the fact that
\begin{equation} \label{invariant eqn}
\vec{\mu}\in V_{p_1^{a_1}} + \dots V_{p_{s}^{a_1}}
\end{equation}
for some $s\leq r$. (As we go along, we will allow ourselves to reorder the primes $p_1,\dots, p_r$.)

Now, let us suppose that at least one of the $p_i^{a_i}$ is $<n$. We 
know that 
$$\vec{\mu}\in \pi^{-1}_{d/p_i^{a_i}}\left(\sum_{p\text{ a prime},p|(d/p_i^{a_i})} W^{(d/p_i^{a_i})}_p\right)$$
and hence, using equation \ref{reduction eqn} and the ongoing assumption
in equation \ref{invariant eqn}, we can either replace the 
term $V_{p_i^{a_i}}$ in the equation \ref{invariant eqn} with $V_{p_i^{a_i+1}}$
or suppress the term entirely (in which case we can reorder the $p_j$s such
that $p_i=p_{s}$, and then replace $s$ with $s-1$, so maintaining the form of
equation \ref{invariant eqn}).

After repeating this procedure a finite number of times (which can be bounded in terms of the 
number of distinct prime powers dividing $d$), we see that
$\vec{\mu}\in V_{p_1^{a_1}} + \dots V_{p_{r'}^{a_1}}$
where each $p_i^{a_i}$ is $\geq n$.  We let $q_1,\dots,q_s$ be the prime powers
$p_1^{a_1},\dots,p_s^{a_a}$.

\smallskip We know then that $\vec{\mu}$ can be written as a sum $v_1+\dots+v_s$
where each $v_i\in V_{q_i}$. On the other hand, this expression is of course not
unique, since the map $\oplus_i V_{q_i}\to\sum_i V_{q_i}$ is not injective. We place
an ordering on tuples $(v_1,\dots,v_s)\in\oplus_i V_{q_i}$ with $\sum_i v_i=\vec{\mu}$
as follows:
\begin{itemize}
\item \emph{First,} we order tuples by the number of initial zeros in the tuple. That is,
we order according to the maximum $t$ such that $v_1,\dots,v_{t-1}$ are all 0. (The more
zeros, the `larger' we consider the tuple to be.)
\item \emph{Second,} amongst tuples with the same number of initial zeros, we order
as follows. Let $v_t$ be the first nonzero element in the tuple. $v_t$ is an element of 
$V_{q_t}$, and hence a function on $\Z/d\Z$. We order according to the number of times
this function takes a nonzero value, with fewer nonzero values counting as making the
tuple `larger'.
\end{itemize}
It is easy to see that we can find a maximal tuple $(v_1,\dots,v_s)$ in this collection.
It is possible that this maximal tuple is in fact $(0,\dots,0)$, whence $\vec{\mu}=0$ and
we are done. So let us suppose that this is not the case, so there is some $t$ with 
$v_1,\dots,v_{t-1}$ all 0 but $v_t$ nonzero.

We write $U$ for $\sum_{i>t} V_{q_iq_t}$, a subset of $V_{q_i}$. We see that we cannot find
a $u\in U$ with $v_t-u$ (an element of $V_{q_t}$, and hence a function on $\Z/d\Z$)
taking fewer nonzero values than $v_t$ does, else we could write 
$u=\sum_{i>t} u_i$ where $u_i\in V_{q_iq_t}\subset V_{q_i}\cap V_{q_t}$, and 
replace the tuple $(0,\dots,0,v_t,v_{t+1},\dots,v_s)$ with 
$(0,\dots,0,v_t-u,v_{t+1}+u_{t+1},\dots,v_s+u_s)$ which is larger according to the
ordering above, contradicting the maximality of $(v_1,\dots,v_s)$.

Now, for each $i\in\Z/(d/q_t)\Z$ we write $S_i$ for the set of elements of $\Z/d\Z$
congruent to $i$ mod $d/q_t$. We see that the number of nonzero values that 
$\vec{\mu}$ takes on the subset $S_t$ must be at least the number of nonzero
values $v_t$ takes on what subset, otherwise we could find a $u$ as in the previous
paragraph. But $v_t$ is nonzero and periodic with period $d/q_i$. Thus $v_t$
is nonzero for at least $q_i>n>\Xi$ possible inputs, so $\vec{\mu}$ is too, 
contradicting the first bullet point in the statement of the lemma.
\end{proof}

\begin{lem} 
Suppose $\Xi$, $\Phi$, and $N$ are positive integers, and $\eps>0$
is a real number. Then we can find an integer $L(N,\Phi,\Xi,\eps)$ with the following property.

Suppose that $l>L(N,\Phi,\Xi)$ is a prime, $\nu<\Phi$ is an integer, and $d\in\Z_{\geq1}$.
Suppose that for each $i=1,\dots, \nu$ we are given a $d_i\in\Z_{\geq1}$ with $d_i|d$. 
Write  $\tilde{T}$ for the finite (additive) group 
$$\tilde{T} = \left\{(t_1,\dots,t_\nu)\in\prod_{i=1,\dots,\nu} \Z/(l^d-1)\Z \quad\vline\quad 
 l^{d_i} t_i = t_i \text{ for all $i$}\right\}.$$

Write $S$ for the set 
$$S=\prod_{i=1,\dots,\nu} \prod_{j=0,\dots,d_i} (\Z\cap [-N,N]),$$
so we write a typical element of $S$ as $\vec{\mu}=((\mu_{1,0},\dots,\mu_{1,d_1}),\dots,
(\mu_{\nu,0},\dots,\mu_{\nu,d_\nu}))$. 
Write $S_\Xi$ for the subset of $S$ consisting of elements
$\mu$ for which $\#\{(i,j)|\mu_{i,j}\neq 0\} < \Xi$.
For $\vec{\mu}\in S$ and $t\in \tilde{T}$, we define
$$\vec{\mu}(t) = \sum_{i=1,\dots,\nu} \sum_{j=0,\dots,d_i} l^j \mu_{i,j} t_i \quad
\quad (\text{so }\vec{\mu}(t)\in \Z/(l^d-1)\Z).$$
Then we have
$$\#\{(\vec{\mu},t)\in S_\Xi\times \tilde{T}| \vec{\mu}(t) = 0\} \leq \eps\,\# \tilde{T}.$$
\end{lem}
\begin{proof}
{\sl Step 1: selecting $L$.} 
We apply Lemma \ref{lem432} with $\Omega$ and $\Xi$ as in the
present context, deducing the existence of an integer $n$ with the property described
there.  We write $Q$ for the set of tuples $(q_1,\dots,q_\nu)$
where each $q_i$ is either a prime power $<n$, or the symbol $\infty$;
and we note that this set is finite. We set
$\delta=\eps/\#Q$.

It then is a trivial exercise in analysis that we can choose $\Omega$
large enough that 
$$\frac{(\Xi^{\Phi})^x (2N)^\Xi}{(\sqrt[n]{\Omega})^{x/\log\log x}} < \delta$$
for all $x\geq 1$. We then apply Lemma \ref{lem431} with this choice of $\Omega$, 
and with $\Xi$ and $N$ as in the present context; then lemma
furnishes a choice of $L$ with a certain property (as described there). This is the $L$
we will use.

Now we suppose that $l,\nu,d$ etc.~are as in the statement of the lemma.

\smallskip
{\sl Step 2: introducing the sets $S(\vec{q})$ for each $\vec{q}=(q_1,\dots,q_\nu)\in Q$, and showing
$\bigcup_{\vec{q}\in Q} S(\vec{q}) \supset S_\Xi$.} For each $\vec{q}\in Q$ with the property
that, for each $i$, \emph{either} $q_i=\infty$, \emph{or} $q_i|d_i$,
we write $S(\vec{q})$ for the subset of $S_\Xi$ consisting of elements $\vec{\mu}$ for
which, for those $i\in\{1,\dots,\nu\}$ where $q_i$ is a prime power, we have that:
$$\sum_{j=0,\dots,d_i-1} \zeta_d^{q_i j} \mu_{i,j} \neq 0$$
while for those $i$ where $q_i$ is $\infty$, we have that $\mu_{i,j}=0$. 
For $\vec{q}\in Q$ with $q_i\nmid d_i$ for some $i$, we set 
$S(\vec{q})=\emptyset$.

By application of Lemma \ref{lem432}, we see that for each
fixed $\vec{\mu}\in S_\Xi$, and each fixed $i$, there either is some prime power $q_i<n$ dividing
$d_i$ such that the sum in the displayed equation is nonzero, or else we have that $\mu_{i,j}$
is zero for all $j$ (this is since $\#\{j|\mu_{i,j}\neq 0\}<\Xi$).
It follows that for each fixed $\vec{\mu}\in S_\Xi$, we can find a tuple $\vec{q}\in Q$ such that
the sum in the displayed equation is nonzero for every $i$, and hence it follows that 
$\bigcup_{\vec{q}\in Q}S(\vec{q}) \supset S_\Xi$.

\smallskip
{\sl Step 3: bounding
$\#\{(\vec{\mu},t)\in S(\vec{q})\times \tilde{T}| \vec{\mu}(t) = 0\}$
for each $\vec{q}$.}
We now fix $\vec{q}\in Q$.
Our goal in this step of the proof is to show 
$\#\{(\vec{\mu},t)\in S(\vec{q})\times \tilde{T}| \vec{\mu}(t) = 0\} \leq \delta\,\# \tilde{T}.$
If $q_i\nmid d_i$ for some $i$, then $S(\vec{q})=\emptyset$ and there is nothing
to prove, so we may assume $q_i|d_i$ or $q_i=\infty$ for all $i$. Moreover, for
those $i$ for which $q_i=\infty$, we can simply imagine that the corresponding $\mu_{i,j}$
(and $t_i$) no longer exist (i.e. $\nu$ is reduced) and then suppress those $q_i$. So
we can assume that $q_i|d_i$ for all $i$.

Let $d_{\max}=\max_i d_i$. To give an element of $S(\vec{q})$, it suffices to
choose $\Xi$ pairs $(i,j)$ (this chooses $\Xi$ of the $\mu_{i,j}$ to be
potentially nonzero), then for each of these $\Xi$ of the $\mu_{i,j}$ we must choose
their value. The first choice may be made in at most $\Xi^{\sum_i d_i}$ ways, and the 
second in fewer than $(2N)^\Xi$ ways. Thus 
$$\#S(\vec{q}) \leq \Xi^{\sum_i d_i} (2N)^\Xi < \Xi^{\Phi d_{\max}} (2N)^\Xi.$$

Next let us fix some $\vec{\mu}\in S(\vec{q})$. Let $i$ be some index such that 
$d_i=d_{\max}$. We may think of $\vec{\mu}$ as
determining a homomorphism $t\mapsto \vec{\mu}(t)$ mapping $\tilde{T}
\to \Z/(l^d-1)\Z$, and we may think of $\vec{\mu}_i=(\mu_{i,0},\dots,\mu_{i,d_i-1})$
as determining a homomorphism $\Z/(l^{d_{\max}}-1)\Z \to \Z/(l^{d_{\max}}-1)\Z,$
 where $t\mapsto \sum_j l^j\mu_{i,j} t$. We have that
\begin{align*}
\#\{t\in \tilde{T}|\vec{\mu}(t)=0\}&=\#\ker \vec{\mu} = \frac{\#\tilde{T}}{(\tilde{T}:\ker\vec{\mu})}
=\frac{\#\tilde{T}}{\#\vec{\mu}(\tilde{T})}\\
&\leq \frac{\#\tilde{T}}{\#\vec{\mu}_i(\Z/(l^{d_{\max}}-1)\Z)}
<\frac{\#\tilde{T}}{(\sqrt[q_i]{\Omega})^{d_{\max}/\log\log d_{\max}}}\\
&<\frac{\#\tilde{T}}{(\sqrt[n]{\Omega})^{d_{\max}/\log\log d_{\max}}}
\end{align*}
(where the penultimate inequality relies on our choice of $l>L$, where $L$ was the constant
from Lemma \ref{lem431}). 

Combining the bounds on $\#S(\vec{q})$ and, for each $\vec{\mu}\in S(\vec{q})$, on
$\#\{t\in \tilde{T}|\vec{\mu}(t)=0\}$, we see that
\begin{align*}
\#\{(\vec{\mu},t)\in S(\vec{q})\times \tilde{T}| \vec{\mu}(t) = 0\} &\leq \frac{\#\tilde{T}}{(\sqrt[n]{\Omega})^{d_{\max}/\log\log d_{\max}}}\Xi^{\Phi d_{\max}} (2N)^\Xi<\delta\,\#\tilde{T}
\end{align*}
completing this step of the proof.

\smallskip
{\sl Step 4: concluding the argument.} We may write, using step 2,
$$\{(\vec{\mu},t)\in S_\Xi\times \tilde{T}| \vec{\mu}(t) = 0\} \subset \bigcup_{\vec{q}\in Q} 
\{(\vec{\mu},t)\in S(\vec{q})\times \tilde{T}| \vec{\mu}(t) = 0\}$$
and hence
\begin{align*}
\#\{(\vec{\mu},t)\in S_\Xi\times \tilde{T}| \vec{\mu}(t) = 0\} &\leq \sum_{\vec{q}\in Q} \#\{(\vec{\mu},t)\in S(\vec{q})\times \tilde{T}| \vec{\mu}(t) = 0\} \\
&\leq \sum_{\vec{q}\in Q} \delta\,\#\tilde{T} = \#Q\delta\,\#\tilde{T} = \eps\,\#\tilde{T}.
\end{align*}
This is as required.
\end{proof}

\begin{lem} \label{regular elements lemma - gm version}
Suppose $\Xi$, $\Phi$, and $N$ are positive integers, and $\eps>0$
is a real number. 
Then we can find an integer $L(N,\Phi,\Xi)$ with the following property.

Suppose that $l>L(N,\Phi,\Xi)$ is a prime, $\nu<\Phi$ is an integer, and that $k/\F_l$ 
is a finite extension. Suppose that for each $i=1,\dots, \nu$ we are given 
a field $k_i$, with $\F_l\subset k_i\subset k$. Write $d_i=[k_i:\F_l]$.
Write $T^*$ for the finite
group $\prod_i \Gm(k_i)$. Write $S$ for the set 
$$S=\prod_{i=1,\dots,\nu} \prod_{j=0,\dots,d_i} (\Z\cap [-N,N]).$$
Write $S_\Xi$ for the subset of $S$ consisting of elements
$\mu$ for which $\#\{(i,j)|\mu_{i,j}\neq 0\} < \Xi$.

For $\vec{\mu}\in S$ and $t\in T^*$, we define
$$\vec{\mu}(t) = \prod_{i=1,\dots,\nu} \prod_{j=0,\dots,d_i} t_i^{l^j \mu_{i,j}} \in k^\times.$$

\noindent Then we have
$$\#\{(\vec{\mu},t)\in S_\Xi\times T^*| \vec{\mu}(t) = 1\} \leq \eps\,\# T^*,$$
and hence, if we have taken $\eps<1$, there is some $t\in T^*$ such that for no $\vec{\mu}\in S_\Xi$ do
we have $\vec{\mu}(t)=1$.
\end{lem}
\begin{proof} The first part of the lemma reduces to the previous one after the application of
some choice of a `discrete log' isomorphism, $\Gm(k)\isoto \Z/(l^{[k:\F_l]}-1)\Z$. The second
is immediate by considering cardinalities.
\end{proof}

\begin{lem} \label{very regular elements lemma}
Suppose that $n$, $M$ and $\Xi'$ are a positive integers. 
There is a constant $C_2(M,n,\Xi')$ with the following property. 
Suppose $l>C_2(M,n)$ is a prime, and we have have an integer $m$,
fields $k_i$, and groups $\cG^{sc}_i$ exhibiting some subgroup $\Gamma<\GL_n(k)$
as very sturdy,
as per Definition \ref{definition of very sturdy}; suppose further that these have the 
additional properties described in Remark \ref{nicely presented very sturdy subgroups}. 
Suppose for each $i$, $\cT/k_i$
is a maximal torus in $\cG_i$. Then we can find elements $g_i\in\cT_i(k_i)\subset\cG_i(k_i)$ for 
each $i$, such that the map
\begin{align*}
\alpha:\prod_{i=1,\dots,m} \left\{ \lambda\in X(\cT_{i,\bar{k}})\,|\,||\lambda||<C_1(n)\right\}^{[k_i,\F_l]} &\to \bar{k}^\times\\
((\lambda_{1,1},\dots,\lambda_{1,[k_1,\F_l]}),\dots,(\lambda_{m,1},\dots,\lambda_{m,[k_m,\F_l]}))&\mapsto (\prod_i \prod_{j=0,\dots,[k_1:\F_l]} \lambda_{i,j}(g_i)^{l^j} )^M
\end{align*}
is injective on the subset of the domain consisting of $\vec{\lambda}$ with at most $\Xi'$
of the $\lambda_{i,j}$ nonzero.
(Here $X(\cT_{i,\bar{k}})$ denotes the set of weights for $\cT_{i,\bar{k}}$.)
\end{lem}
\begin{proof}
Before we choose the constant $C_2(M,n)$, let us imagine briefly that we 
have an integer $m$, fields $k_i$, and groups $\cG^{sc}_i$ exhibiting some 
subgroup $\Gamma<\GL_n(k)$ as very sturdy in order to introduce some notation.
As per Remark \ref{nicely presented very sturdy subgroups}, for each $i$, $\cT_i(k_i)$ can be written
as $\prod_{o\in\cO_i} \Gm(k_{i,o})$, where $\cO_i$ is the set of orbits of the simple roots
in $X(\cT_i)$ under $\Gal(\bar{k}/k_i)$, $o$ stands for a particular orbit, and $k_{i,o}$ is 
an extension of $k_i$ with $[k_i:k_{i,o}]=\#o$. We can associate to a tuple
$(\lambda_{i,1},\dots,\lambda_{i,[k_i,\F_l]})$ a collection of tuples
$(\mu_{i,o,1},\dots,\mu_{i,o,[k_{i,o}:\F_l]})$ such that, whenever 
\begin{itemize}
\item $(g_1,\dots,g_m)$ is a tuple
of elements in $\prod_i\cT_i(k_i)$, which corresponds to a tuple of elements
$$((t_{1,1},\dots,t_{1,\#\cO_1}),(t_{2,1},\dots,t_{2,\#\cO_2}),\dots, 
(t_{m,1},\dots,t_{m,\#\cO_m})\in \prod_{i=1,\dots,m} \prod_{o\in \cO_i} \Gm(k_{i,o})$$
under the isomorphism $\cT_i(k_i)\cong\prod_{o\in\cO_i} \Gm(k_{i,o})$,
\end{itemize}
we have that
$$\prod_{i=1,\dots,m} \prod_{j=0,\dots,[k_1:\F_l]} \lambda_{i,j}(g_i)^{M l^j} =
\prod_{i=1,\dots,m} \prod_{o\in\cO_i}\prod_{j=0,\dots,[k_{i,o}:\F_l]} t_{i,j}^{\mu_{i,o,j}l^j}$$
Finally, we write $\nu=\sum_i \#\cO_i$. 

At this point we will proceed to choose $C_2(M,n)$, which must not depend on
the integer $m$, fields $k_i$, and groups $\cG^{sc}_i$.
The fact that each $\lambda_{i,j}$ has $||\lambda_{i,j}||<C_1(n)$ allows us to find
an integer $N$ depending only on $M$ and $n$ such that whenever we have $\mu_{i,o,j}$s
as above, $-N<2 \mu_{i,o,j}<N$ for all $i,o,j$. 

We can also bound in terms of $n$ the largest possible rank of any of
the groups $\cG_i$, (since the group must have a faithful mod center 
representation of dimension $<n$). Thus, if we assume at most $2\Xi'$ of the $\lambda$
are nonzero, since we can bound how many $\mu$s are associated to each of these
$\lambda$ (using the bound on the rank of the $\cG_i$),
and because only the $\mu$s associated to a nonzero $\lambda$
may possibly be nonzero, we deduce that there is a constant $\Xi$ depending
on $n$ and $\Xi'$ alone such that at most $\Xi$ of the $\mu$s are nonzero.

Similarly, we can bound $\nu$ in terms of $n$ alone, since we first bound $m$
in terms of $n$ and then use the bound on the ranks of the $\cG_i$ to bound
the $\#\cO_i$. Let $\Phi$ be the upper bound on $\nu$.

We apply Lemma \ref{regular elements lemma - gm version} with
\begin{itemize}
\item $\Phi$, $\Xi$ and $N$ as in the present context, and
\item $\eps=0.9$.
\end{itemize}
We get a constant $L$, and we put $C_2(M,n,\Xi')=L$.

We then see that if we 
have an integer $m$, fields $k_i$, and groups $\cG^{sc}_i$ exhibiting some 
subgroup $\Gamma<\GL_n(k)$ as very sturdy, and we apply the property of
the bound $L$ given in Lemma \ref{regular elements lemma - gm version},
we have that we can find a $t\in \prod_i \prod_{o\in\cO_i}\Gm(k_{i,o})$
(corresponding to an element $(g_1,\dots,g_m)$, in
$\prod_i \cT_i(k_i)\subset \prod_i \cG_i(k_i)$, with the following property:
\begin{itemize}
\item For any $\lambda\in\prod_{i=1,\dots,m} \left\{ \lambda\in X(\cT_{i,\bar{k}})\,|\,||\lambda||<C_1(n)\right\}^{[k_i,\F_l]}$ with at most $2\Xi'$ of the $\lambda_{i,j}$ nonzero (and hence with at most $\Xi$ of the corresponding
$\mu$s nonzero) we have that 
$$\prod_{i=1,\dots,m} \prod_{j=0,\dots,[k_1:\F_l]} \lambda_{i,j}(g_i)^{M l^j} =
\prod_{i=1,\dots,m} \prod_{o\in\cO_i}\prod_{j=0,\dots,[k_{i,o}:\F_l]} t_{i,j}^{\mu_{i,o,j}l^j}$$
is never 1.
\end{itemize}
It follows that the map $\alpha$ given in the statement of the lemma is injective on the subset
of its domain as described there, if $\vec{\lambda}$, $\vec{\lambda}'$ are two elements
in the domain with at most $\Xi'$ of the $\lambda_{i,j}$ and $\lambda'_{i,j}$ nonzero,
then at most $2\Xi'$ of the $(\lambda/\lambda')_{i,j}$ are nonzero, and so
$\alpha(\lambda/\lambda')$ cannot equal 1, by the bullet point immediately above.
\end{proof}

\subsection{Sturdy implies big}

Our next goal is to show that sturdy subgroups are automatically big, at least 
for $l$ large. Our arguments draw will very heavily on \cite{snowden2009bigness,pjw}.

\begin{prop} \label{very sturdy subgroups are big}
Let $M$ and $n$ be positive integers. There is a 
constant $C_3(M,n)$ depending only on $M$ and $n$ with the following property:
if $l$ is a prime number which is larger than $C_3(M,n)$, $k/\F_l$ 
a finite extension, and $\Gamma<\GL_n(k)$ a very sturdy subgroup, then $\Gamma$ is $M$-big.
\end{prop}
\begin{proof} Since $\Gamma$ is very sturdy, it is abstractly isomorphic to a product of
finite simple groups, each of which is a simple group $\cD(\cH(\F_q))$ of Lie type with $l|q$.
(See Lemma \ref{sturdy subgps up to iso}.) Any normal subgroup must just be a product of
a subset of the factors, and so the quotient will isomorphic to the product of the complimentary
factors, and hence not of $l$ power order. This gives us the first bullet point in the definition
of $M$-bigness.

Since $\Gamma$ is very sturdy, it is assumed to act absolutely irreducibly, and the second bullet
point in the definition of bigness follows.

In order to discuss the remaining points, we will let $m$, $\cG_1,\dots,\cG_m$, $k_1,\dots,k_m$
and $r_1,\dots,r_m$ refer to the various objects of those names described in the definition
of very sturdy, Definition \ref{definition of very sturdy}. We can and will assume that these have been 
chosen with the additional useful properties as described in 
Remark \ref{nicely presented very sturdy subgroups}.
By definition, $r_i$ is a faithful projective representation over $k$ of $\pi_i(\cG^{sc}(k_i))$;
we will write $V_i$ for this representation space\footnote{A convention which we will
follow in the present proof is to use Roman letters like $V$ for representation spaces
of abstract finite groups, and cursive letters like $\cV$ for algebraic representations.}, 
which must be absolutely irreducible 
since otherwise $\Gamma$ would fail to act absolutely irreducibly, contradicting a 
hypothesis of sturdiness. 

We now turn to proving the third bullet point in the definition of bigness.
As explained in \cite[1.13]{larsen1995maximality}, every $k$
representation of $\cG^{sc}_i(k_i)$ is a direct sum of irreducible representations over $\bar{k}$,
and hence every self-extension of $V_i$ is trivial.
It follows that $H^1(\cG^{sc}_i(k_i), \ad V_i)=(0)$. From this and the fact that $\ad V_i$ is
semisimple (again from \cite[1.13]{larsen1995maximality}), so $\ad^0 V_i$ is a direct summand, we see
$H^1(\cG^{sc}_i(k_i), \ad^0 V_i)=(0)$.
Then $H^1(\cG^{sc}_i(k_i), \ad^0 V_i\otimes_{k_i} k)=(0)$, and hence:
\begin{align*}
H^1&(\prod_i \cG^{sc}_i(k_i),  \ad^0 \bigotimes_i (V_i\otimes_{k_i} k))\\
\displaybreak[0]
&=H^1(\prod_i \cG^{sc}_i(k_i),  \bigoplus_{S\subset \{1,\dots m\},S\neq\emptyset} \bigotimes_{i\in S} \ad^0 (V_i\otimes_{k_i} k))\\
\displaybreak[0]
&=\bigoplus_{S\subset \{1,\dots m\},S\neq\emptyset} H^1(\prod_i \cG^{sc}_i(k_i),  \bigotimes_{i\in S} \ad^0 (V_i\otimes_{k_i} k))\\
\displaybreak[0]
&=\bigoplus_{S\subset \{1,\dots m\},S\neq\emptyset} \left(\bigotimes_{i\in S} H^1( \cG^{sc}_i(k_i),  \ad^0 (V_i\otimes_{k_i} k))\right)\otimes\left(\bigotimes_{i\not\in S}\Hom(\cG^{sc}_i(k_i),k)\right)\\
&=\bigoplus_{S\subset \{1,\dots m\},S\neq\emptyset} (0) = (0)
\end{align*}
But this tells us $H^1(\prod_i \cG^{sc}_i(k_i), \ad k^n)=(0)$, where $\prod_i \cG^{sc}_i(k_i)$
acts on $k^n$ via $r$. 

Let $K$ be the kernel of $r$. We have an exact sequence 
$$1\to K \to \prod_i \cG^{sc}_i(k_i) \to r(\prod_i \cG^{sc}_i(k_i)) \to 1$$
and hence an injection
$$H^1(r(\prod_i \cG^{sc}_i(k_i)),(\ad^0 k^n)^K) \into H^1(\prod_i \cG^{sc}_i(k_i),\ad^0 k^n)$$
and so since the group on the right vanishes, so does the group on the left.  Since under $r$
we have that $K$ acts trivially on $k^n$, it acts trivially on $\ad^0 k^n$. Thus 
$H^1(r(\prod_i \cG^{sc}_i(k_i)),\ad^0 k^n) =(0)$,
and the third bullet point in the definition of bigness holds for $r(\prod_i \cG^{sc}_i(k_i))$. 
It follows that it holds for $\Gamma$, arguing as in \cite[Prop 2.2]{snowden2009bigness}.

All that remains is the fourth bullet point in the definition of bigness. Before we turn to this,
we would first like to relate the $r_i$ to algebraic representations. Fix some $i$, $1\leq i\leq m$. 
We now apply Theorem \ref{steinberg-rep-thm} with:
\begin{itemize}
\item $k$ there being our current $k_i$, and
\item $\cG$ there being our current $\cG_i^{sc}$,
\end{itemize}
we see that $V_i\otimes \bar{k}$ (which is initially a representation of $\pi(\cG_i(k_i))$,
but can therefore be thought of a as a representation of $\cG_i(k_i)$ by composing with $\pi$)
is of the form $\bigotimes_{j=0}^{[k:\F_l]-1} \cV_{i,j}^{\Frob^i}$,
where each $\cV_{i,j}$ is an irreducible algebraic representation of $\cG_i/k_i$ with highest weight 
$\lambda$ satisfying $0\leq \lambda(\mathbf{a}) \leq l-1$, and with $k_i$ coefficients. 
By applying Corollary \ref{norm-bounding-corollary}, we can bound the norm
of $\cV_{i,j}$, deducing that $||\cV_{i,j}||<C_1(n)$, where $C_1(n)$ is the constant from 
Corollary \ref{norm-bounding-corollary}.

Now, for each $i$, recall we have chosen (as per Remark \ref{nicely presented very sturdy subgroups}) 
a Borel subgroup $\cB_i$ of $\cG_i$ defined over $k_i$,
and a maximal torus $\cT_i$ for $\cB_i$ (automatically a maximal torus for $\cG_i$).
Let $\cV_{i,j,0}$ be
$\cV^{\cU_i}_{i,j}$, where $\cU_i$ is the unipotent radical of $\cB_i$. Let 
$\lambda_{i,j}:\cT_i\to\mathbb{G}_m$ give the action of $\cT_i$ on $\cV_{i,j,0}$;
as in the \cite{snowden2009bigness} (the second paragraph before the proof of 
Lemma 5.2 there), we see that $\lambda_{i,j}$ is a highest weight of $\cV_{i,j}$
and occurs as a weight with multiplicity 1. 

Let $e_{i,j}$ be a vector in $\cV_{i,j}$,
then let $e_i=\bigotimes_j e_{i,j}$, a vector in 
$\bigotimes_j \cV_{i,j}$. Since $V_i$ corresponds to $\otimes_j \cV_{i,j}^{\Frob j}$
under the correspondence of Theorem \ref{steinberg-rep-thm}, and in particular
they have the same underlying vector space over $k$, we can think of this
$e_i$ also as a vector in $V_i$. Finally let $e=\bigotimes_i e_i$, a vector in a 
representation space for $r$, $V$ say. We see, by Lemma \ref{very regular elements lemma}
that we can find an element $g\in\prod_i\cG_i(k_i)$ such that $e$ is an eigenvector
of $g$ whose corresponding eigenvalue $\alpha$ is a simple root of the characteristic 
polynomial of $g|V$, and which indeed has the property that any other root $\beta$
of this polynomial has $\alpha^M\neq \beta^M$.

On the other hand, by \cite[Lemma 5.2]{snowden2009bigness}, and the arguments
immediately before we see that every
nonzero irreducible submodule of $\ad (\cV_{i,j})$ has non-zero projection onto 
$\ad (\cV_{i,j,0})$; whence every nonzero irreducible submodule of 
$\ad (\bigotimes_j \cV^{\Frob^j}_{i,j})$ has nonzero projection onto 
$\ad (\bigotimes_j \cV^{\Frob^j}_{i,j,0})$; and thus (using Theorem \ref{steinberg-rep-thm},
Theorem \ref{larsen-ss-thm}, and \mar{REF}) we see that 
every nonzero irreducible submodule of $V_i\otimes \bar{k}$ has nonzero projection onto 
$\langle e_i\rangle\otimes \bar{k}$. It follows every nonzero irreducible submodule of 
$V_i$ has nonzero projection onto  $\langle e_i\rangle\otimes \bar{k}$.
Thus every nonzero irreducible submodule of $V$ has
nonzero projection onto $\langle e\rangle$. This is as required.
\end{proof}

\begin{cor} \label{sturdy subgroups are big} Let $M$ and $n$ be positive integers. There is a 
constant $C_3(M,n)$ depending only on $M$ and $n$ with the following property:
if $l$ is a prime number which is larger than $C_3(M,n)$, $k/\F_l$ 
a finite extension, and $\Gamma<\GL_n(k)$ a sturdy subgroup, then $\Gamma$ is $M$-big.
\end{cor}
\begin{proof} We take the constant $C_3(M,n)$ to be as in the Proposition. Then, given a sturdy
$\Gamma$, we look at the chain of subgroups $\Gamma_1\triangleleft \Gamma_2 \triangleleft\dots
\triangleleft \Gamma_r=\Gamma$. We apply the Proposition to see that $\Gamma_1$ is $M$-big; 
we then apply Proposition \ref{normal subgroups and bigness} to see inductively that each
$\Gamma_i$, $i>1$ is $M$-big. In particular, $\Gamma=\Gamma_r$ is $M$-big.
\end{proof}

\section{The main result}\label{sec: main result}

\subsection{} \label{ss: main result}
The heart of what we will prove is the following rather technical proposition.
\begin{prop} \label{main prop-groups version}
For each positive integer $n$, there is an integer $A_n$ with the following
property. Let $l>A_n$ be a prime, $k/\F_l$ a finite extension with $l\nmid[k:\F_l]$, 
and $\Gamma<\GL_n(k)$ 
a subgroup containing $k^\times$. For convenience of notation, let us also
choose a number field $L$ and prime $\lambda$ of $L$ such that 
$\bigO_L/\lambda \bigO_L=k$.
Then one of the following must occur:
\begin{enumerate}
\item $\Gamma$ does not act absolutely irreducibly on $k^n$.
\item $\Gamma$ lies inside some imprimitive subgroup 
$G_{V_1\oplus\dots\oplus V_m}$ (see Definition \ref{def: Aschbacher Dynkin subgps});
that is, $\Gamma$ preserves
a direct sum decomposition, $k^n=V_1\oplus\dots \oplus V_m$ (where $\dim V_i=n/m$
for all $i$), though it need not preserve the individual terms in the direct sum.
\item We can find a $k$ vector space $V_k$, an $L_\lambda$ vector space $V_L$,
an $\bigO_{M_\lambda}$ lattice $\Lambda\subset V_L$, a finite subgroup $G<\GL(\Lambda)$,
and an isomorphism $k^n\cong V_k\otimes \overline{V_L}$ (where $\overline{V_L}$ is the $k$-vector
space $\Lambda\otimes_{\bigO_{L_\lambda}} k$),
such that we can factor the map $\Gamma\into\GL(k^n)\onto\PGL(k^n)$ through the map
$$\PGL(V_k)\times G \into
\PGL(V_k)\times\GL(\Lambda) \onto
\PGL(V_k)\times\PGL(\Lambda) \onto
\PGL(V_k)\times\PGL(\overline{V_L}) \into \PGL(k^n)$$
\item $\Gamma$ is a sturdy subgroup, in the sense of Definition \ref{definition of sturdy}.
\end{enumerate}
Furthermore, in case (3) the order of the group $G$ can be bounded in terms of $n$.
\end{prop}

\begin{proof}
We prove the claim by induction on $n$. We may therefore inductively suppose the
existence of $A_i$ for all $i<n$. 

By applying Proposition \ref{Larsen-prop constructing set S} with $\cG$ being the
algebraic group $\mathcal{PGL}$, we see that there exists a constant $X$ and a 
finite set $S$ of almost simple groups, (a group $H$ is \emph{almost simple}
if $G< H < \Aut G$ for some simple group $G$) such that for all $l>X$ and finite 
extensions $k/\F_l$ we have that 
\begin{itemize}
\item Every subgroup of $\PGL(\F_l)$ which is a finite almost simple group is either
isomorphic to a member of $S$ or to an almost simple group whose corresponding 
simple group is a derived group of an adjoint group of Lie type, $\cD(\cH(\F_q))$, 
where $l|q$ and $\cD(\cH(\F_q))=\text{Im}(\cH^{sc}(\F_q)\to\cH(\F_q))$.\labelbp{subgroup-S}
\end{itemize}
We can and do choose $A_n$ to be large enough that, if $l$ is a prime larger than $A_n$:
\begin{itemize}
\item $l>X$.\labelbp{l-bigger-than-D}
\item $l>A_i$, $i=1,\dots,n-1$.\labelbp{large-enough-to-apply-ind-hyp}
\item If $n=p^m$ is an odd prime power, then $l$ is coprime to $\#p_+^{1+2m}$ and to $\#\Sp_{2m}(\F_p)$.
\labelbp{coprime-to--order-of-odd-extraspecial-group}
\item If $n=2^m$ is a power of 2, then $l>2$ (and hence is coprime $\#2_+^{1+2m}$ and $\#2_-^{1+2m}$), 
and furthermore $l$ is coprime to both $\#\GO^+_{2m}(\F_2)$ and $\#\GO^-_{2m}(\F_2)$.
\labelbp{coprime-to--order-of-even-extraspecial-group}
\item $l$ is coprime to the orders of all the groups in the set $S$.
\labelbp{coprime-to--order-of-S-groups}
\item We have $l \nmid m$ for all $m\leq n$.\labelbp{l-does-not-divide-nums-up-to-n}
\end{itemize}
We must now show that this $A_n$ will have the property described in the proposition.

We consider first the special case where $\Gamma$ is all of $\GL_n(k)$. In this case $\Gamma$
contains $\SL_n(k) k^\times$ as a normal subgroup with prime-to-$l$ index, while $\SL_n(k) k^\times$
is clearly very sturdy, taking $\cG$ as the restriction of scalars of $\cSL_n$ from $k$ to $\F_l$
(this is manifestly simply connected and semisimple),
and the obvious map $\cG\to\cSL_{n[k:\F_l]}$ (which is obviously algebraic). Thus $\Gamma$
is in this case sturdy.

Otherwise, we apply the Aschbacher-Dynkin Theorem (Theorem 
\ref{thm: Aschbacher Dynkin thm}), to $\Gamma$. We see 
that $\Gamma$ is contained in some subgroup of $\GL_n(k)$
of one of the kinds described in Definition \ref{def: Aschbacher Dynkin subgps}, (1)--(6). 
The rest of our proof will be broken into cases, depending on in which kind of 
subgroup $\Gamma$ is contained. 

\begin{itemize}
\item {\sl Case 1: A reducible subgroup.} In this case, we see that $\pi^{-1}(\Gamma)$ acts 
reducibly, and we are done (alternative 1 of the theorem to be proved).
\item {\sl Case 2: An imprimitive subgroup, $G_{V_1\oplus\dots V_m}$.} In this case, 
we are also done (alternative 2 of the theorem to be proved).
\item {\sl Case 3: The stabilizer of a binary tensor product decomposition, factor by factor.} 
Suppose $\Gamma$ stabilizes a tensor product decomposition $V_1\otimes V_2$. 

In this case, the map $\Gamma\into \GL_n(k)\onto \PGL_n(k)$ will factor through the 
obvious map $\PGL(V_1)\times\PGL(V_2)\into\PGL(k^n)$, giving a map $
\phi:\Gamma\to \PGL(V_1)\times\PGL(V_2)$. Let $\pi_i:\PGL(V_1)\times\PGL(V_2)
\to\PGL(V_i)$ ($i$=1,2) denote the projection, and let $q_i:\GL(V_i)\to\PGL(V_i)$ 
($i$=1,2) denote the natural quotient map. Finally, let $\Gamma_i:=
q_i^{-1}(\pi^{-1}_i(\phi_i(\Gamma)))$.

By point \ref{large-enough-to-apply-ind-hyp}, $l>A_{\dim V_1}$, and we may apply
our inductive hypothesis to yield that $\Gamma_1$ satisfies one of the alternatives (1--4) in
the statement of the proposition. We first show that we are done if any of the first three
alternatives hold.
\begin{itemize}
\item {\sl Case 3a: $\Gamma_1$ does not act absolutely irreducibly.} In this case, after 
an extension of fields $k'/k$, $V_1\otimes_k k'$ has a nontrivial proper subspace 
$W$ which is $\Gamma_1$ stable. Then $(k')^n$ has a nontrivial proper subspace 
$W \otimes_{k'} (V_2\otimes_k k')$ which is $\Gamma$ stable, and we have alternative
1 in the statement of the present proposition.
\item {\sl Case 3b: $\Gamma_1$ acts imprimitively.} In this case, $\Gamma_1$ will preserve a direct sum
decomposition $V_1=W_1\oplus\dots\oplus W_m$ for some $m$, where 
$\dim W_i=\dim V_1/m$ for all $i$. Then $\Gamma$ will preserve the direct sum decomposition
$W_1\otimes V_2\oplus\dots W_m\otimes V_2$, and we have alternative
2 in the statement of the present proposition.
\item {\sl Case 3c: The map $\Gamma_1\into\GL(V_1)\onto\PGL(V_1)$ factors through}
\begin{align*}
\quad\quad \quad\quad\PGL(V_{1,k})\times G &\into
\PGL(V_{1,k})\times\GL(\Lambda) \onto
\PGL(V_{1,k})\times\PGL(\Lambda) \\
&\onto
\PGL(V_{1,k})\times\PGL(\overline{V_{1,L}}) \into \PGL(V_1).
\end{align*}
(Here we have a $k$ vector space $V_k$, an $L_\lambda$ vector space $V_L$,
an $\bigO_{L_\lambda}$ lattice $\Lambda\subset V_L$, a finite subgroup $G<\GL(\Lambda)$,
and an isomorphism $k^n\cong V_k\otimes \overline{V_L}$.)
In this case, we have $V\cong V_1\otimes V_2\cong V_k \otimes \overline{V_{1,L}}$ where
$V_k:=V_2\otimes V_{1,k}$, and a commutative diagram as in Figure \ref{big-comm-diag}.
\begin{figure}
$$\xymatrix{
\GL(V)\ar@{->>}[drr]
\\
&\PGL(V_1)\times\PGL(V_2)\ar@{^{(}->}[r] & \PGL(V)
\\
\Gamma\ar@/_3pc/[ddr]\ar@{^{(}->}[uu]\ar@{^{(}->}[ur]
&(\PGL(V_{1,k})\times \PGL(\overline{V_{1,L}}))\times\PGL(V_2)\ar@{^{(}->}[r]\ar@{^{(}->}[u]
    & (\PGL(V_{k})\times \PGL(\overline{V_{1,L}}))\ar@{^{(}->}[u]
\\
&(\PGL(V_{1,k})\times \PGL(\Lambda)) \times\PGL(V_2)\ar@{^{(}->}[r]\ar@{->>}[u] 
       & \PGL(V_{1,k})\times \PGL(\Lambda)\ar@{->>}[u] 
\\
&\PGL(V_{1,k})\times G \times\PGL(V_2)\ar[u]\ar@{^{(}->}[r] & \PGL(V_{k})\times G\ar[u]
}$$
\begin{caption}{Commutative diagram for the proof of Proposition \ref{main prop-groups version}. 
\label{big-comm-diag}}
\end{caption}
\end{figure}

Thus we have alternative 3 in the statement of the present proposition.
\end{itemize}
Thus we have reduced to the case where alternative (4) in the statement of the 
Proposition holds for $\Gamma_1$. We may similarly reduce to the case where (4) holds for
$\Gamma_2$. That is, we see that $\Gamma_1$ and $\Gamma_2$ are both sturdy.
We may also assume that $\Gamma$ acts absolutely irreducibly (else we are done, via
the first alternative). 
By applying Lemma \ref{tensor product of sturdy subgroups sturdy}, we see that $\Gamma$
is sturdy, and we have the fourth alternative in the proposition to be proved.
\item {\sl Case 4: The stabilizer of a tensor product decomposition, 
$k^n=V_1\otimes\dots\otimes V_m$  not necessarily 
factor by factor.} In this case, it is easy to see that for each $\gamma\in\Gamma$, 
$\gamma$ determines a permutation of the $V_i$, and we have
a homomorphism $\phi:\Gamma\to S_X$, where $X=\{V_1,\dots,V_m\}$;
that is, $\Gamma$ acts on $X$. 
If this action is not transitive, then 
we may write $V$ as a tensor product with $\Gamma$ preserving the individual terms in the tensor product, and
hence reduce to the previous case. Thus we may assume that $\Gamma$ acts
transitively on $X$.

We consider, $\Stab_\Gamma V_1$, the stabilizer of $V_1$ in $\Gamma$. Then there is a natural map
$\pi:\Stab_\Gamma V_1\to\PGL(V_1)$. Let $q:\GL(V_1)\to\PGL(V_1)$ denote the quotient map,
and let $\Gamma_1=q^{-1}(\pi(\Stab_\Gamma V_1))$.

By point \ref{large-enough-to-apply-ind-hyp}, $l>A_{\dim V_1}$, and we may apply
our inductive hypothesis to yield that $\Gamma_1$ satisfies one of the alternatives (1--4) in
the statement of the proposition. We analyze each of these cases in turn.

\begin{itemize}
\item {\sl Case 4a: $\Gamma_1$ does not act absolutely irreducibly.} In this case, after 
an extension of fields $k'/k$, $V_1\otimes_k k'$ has a nontrivial proper subspace 
$W$ which is $\Gamma_1$ stable. Then $(k')^n$ has the nontrivial proper subspace 
$W \otimes_{k'} W\otimes_{k'} \dots\otimes_{k'} W$ which is $\Gamma$ stable, and we 
have alternative 1 in the statement of the present proposition.
\item {\sl Case 4b: $\Gamma_1$ acts imprimitively.} In this case, $\Gamma_1$ will 
preserve a direct sum decomposition $V_1=W_1\oplus\dots\oplus W_r$ for some $r$, where 
$\dim W_i=\dim V_1/r$ for all $i$. Then $\Gamma$ will preserve the direct sum decomposition
$V\cong \bigoplus_{(i_1,\dots,i_m)} W_{i_1}\otimes\dots \otimes W_{i_m}$
(though not necessarily the individual terms within it) and we have alternative
2 in the statement of the present proposition.
\item {\sl Case 4c: The map $\Gamma_1\into\GL(V_1)\onto\PGL(V_1)$ factors through}
\begin{align*}
\quad\quad \quad\quad\PGL(V_{1,k})\times G_1 &\into
\PGL(V_{1,k})\times\GL(\Lambda) \onto
\PGL(V_{1,k})\times\PGL(\Lambda) \\
&\onto
\PGL(V_{1,k})\times\PGL(\overline{V_{1,L}}) \into \PGL(V_1).
\end{align*}
(Here we have a $k$ vector space $V_k$, an $L_\lambda$ vector space $V_L$,
an $\bigO_{L_\lambda}$ lattice $\Lambda\subset V_L$, a finite subgroup $G_1<\GL(\Lambda)$,
and an isomorphism $k^n\cong V_k\otimes \overline{V_L}$.)

This data tells us that we can think of $V_{1,L}$ as a characteristic zero representation
of $\Gamma_1$, with stable lattice $\Lambda$ and reduction mod $l$ $\overline{V_{1,L}}$,
and think of $\Pj V_{1,k}$ as a projective representation of $\Gamma_1$; and that 
moreover, if we think of $V_1$ as a representation of $\Gamma_1$ in the obvious way, then
the associated projective representation satisfies 
$\Pj V_1\cong \Pj\overline{V_{1,L}} \otimes \Pj V_{1,k}.$

Then we see
\begin{align*}
\Pj V&\cong \Pj(\otimesind_{\Gamma}^{\Gamma_1} V_1) \cong \otimesind_{\Gamma}^{\Gamma_1} \Pj V_1\\
\displaybreak[0]
&\cong \otimesind_{\Gamma}^{\Gamma_1}\Pj\overline{V_{1,L}} \otimes \Pj V_{1,k}\\
\displaybreak[0]
&\cong
(\otimesind_{\Gamma}^{\Gamma_1} \Pj\overline{V_{1,L}}) \otimes (\otimesind_{\Gamma}^{\Gamma_1} \Pj V_{1,k})\\
\displaybreak[0]
&\cong \Pj\overline{\otimesind_{\Gamma}^{\Gamma_1} V_{1,L}} \otimes (\otimesind_{\Gamma}^{\Gamma_1} \Pj V_{1,k})\\
&\cong \Pj\overline{V_{L}}\otimes \Pj V_{k}
\end{align*}
Here $V_{L}:=\otimesind_{\Gamma}^{\Gamma_1} V_{1,L}$ is a characteristic 0 representation
of $\Gamma$, with an invariant lattive $\Lambda^{\otimes m}$
and $\Pj V_k:=\otimesind_{\Gamma}^{\Gamma_1} \Pj V_{1,k}$ is a characteristic $l$
projective representation. These furnish us with a map $\Gamma\to\GL(\Lambda^{\otimes m})$
(with finite image, $G$ say) and a map $\Gamma\to\PGL(V_{1,k})$. The isomorphisms above
then tell us that we can factor $\Gamma\into\GL(V)\onto\PGL(V)$ through 
\begin{align*}
\PGL(V_{k})\times G \into
\PGL(V_{k})\times\GL(\Lambda^{\otimes m}) \onto
\PGL(V_{k})\times\PGL(\overline{V_{L}}) \into \PGL(V).
\end{align*}

Thus we have alternative 3 in the statement of the present proposition.
\item {\sl Case 4d: $\Gamma_1$ is sturdy.} In this case, we introduce a further subgroup
$N:=\ker\phi$ of $\Gamma$, the subgroup acting trivially on $X$. $N$ is normal in 
$\Gamma$, with index dividing $m!$, and so if we put 
$N_1=q^{-1}(\pi(N\cap\Stab_\Gamma V_1))=q^{-1}(\pi(N))$, then $N_1$ is 
normal in $\Gamma_1$ with index dividing $m!)$. In particular, by 
point \ref{l-does-not-divide-nums-up-to-n}, and the obvious fact that $m<n$,
we see $(l,m)=1$, so by Lemma \ref{normal subgp of sturdy of prime to l index is sturdy}, 
we see that $N_1$ is sturdy.

We can construct, by analogy with $\Gamma_1$, further subgroups $\Gamma_2,\dots,\Gamma_m$.
Since, given any element $\gamma_1$ of $\Gamma_1$, we can find an element of 
$\Gamma_2$ which acts on $V_2$ in the same way as $\gamma_1$ acted on $V_1$
(by conjugating by an element of $\Gamma$ which acts to move $V_1$ to $V_2$,
possible since the action of $\Gamma$ on $X$ is transitive), we see 
$\Gamma_2,\dots,\Gamma_m$ are also big. We construct $N_2,\dots,N_m$ by analogy
with $N_1$ above, and we see that each is sturdy, by the same argument. It follows from
Lemma \ref{tensor product of sturdy subgroups sturdy}, applied repeatedly, 
that $N$ is sturdy. Then since
$N\triangleleft \Gamma$ with index prime to $l$ (since dividing $m!$---see point 
\ref{l-does-not-divide-nums-up-to-n}), we conclude $\Gamma$ is sturdy. 

Thus we are done, having the fourth alternative in the proposition to be proved.
\end{itemize}

\item {\sl Case 5: A subgroup of extraspecial type, $G_{p^{1+2m}}$.} Let us first suppose 
$p\neq 2$. Let $G$ be the subgroup of $\GL_n(\Qbar)$ isomorphic to $p_+^{1+2m}$
as constructed in \S3.10.2 of \cite{wilsonfinitesimple}, $N$ its normalizer in $\GL_n(\Qbar)$
(so $N\cong p_+^{1+2m}\rtimes \Sp_{2m}(\F_p)$.) Let $\bar{G}$ and $\bar{N}$ be their
reductions mod $l$, subgroups of $\GL_n(k)$. We have that
$\#\Gamma|\#G_{p^{1+2m}}$ and $\#G_{p^{1+2m}}|\#\bar{N} \#k^\times$, while
$\#\bar{N}|\#N=\#p_+^{1+2m}\rtimes \Sp_{2m}(\F_p)$. Thus 
$$\#\Gamma|(\#k^\times)(\#p_+^{1+2m})(\#\Sp_{2m}(\F_p))$$
and since on the other hand $1=(\#k^\times,l)=(p,l)=(\#\Sp_{2m}(\F_p),l)$ (by point
\ref{coprime-to--order-of-odd-extraspecial-group}), we see that $(\#\Gamma,l)=1$.
We may think of $\Gamma\into\GL_n(k)$ as being a characteristic $l$ representation
of $\Gamma$; since $(\#\Gamma,l)=1$, this lifts to a characteristic zero representation
$r:\Gamma\to\GL(\Lambda)$ where $\Lambda\subset L_\lambda^n$ is a lattice.

We then have alternative 3 of the present proposition, taking $V_k=k$, $G=r(\Gamma)$,
and mapping $\Gamma\to\PGL(V_k)\times G$ via $(1,r)$, where $1$ is the constant
function taking value the identity.

The case that $p=2$ is completely analogous, using \ref{coprime-to--order-of-even-extraspecial-group}.

\item {\sl Case 6: $\Gamma/k^\times$ is contained in the image of an injective
homomorphism $\phi:H\into\PGL_n(k)$, where $H$ is an almost simple group.}

We first consider the case where $H$ is isomorphic to one of the groups in the set $S$
constructed at the beginning of the proof. In such a case, say $\Gamma/k^\times\cong G_0$, $G_0\in S$,
$\#\Gamma=\#(\Gamma/k^\times)\,\#k^\times=\#G_0\,\#k^\times$, so since $(\#k^\times,l)=
(\#G_0,l)=1$ (using point \ref{coprime-to--order-of-S-groups} above), we have 
$(\#\Gamma,l)=1$. Thus, thinking of the inclusion $\Gamma\into\GL_n(k)$ as a representation
of $\Gamma$ in characteristic $l$, we can lift the representation to characteristic zero, and 
hence deduce that alternative 3 of the Proposition holds, just as in case 5 analyzed above.

Thus we may assume on the one hand that $\Gamma/k^\times$ is isomorphic to an almost 
simple group $H$; and an other hand, this almost simple group cannot be isomorphic to any 
group in the set $S$. But then, given point \ref{subgroup-S} above, we see that 
$\Gamma/k^\times$ is isomorphic to an almost simple group 
whose corresponding simple group is a derived group of an adjoint group of Lie type, 
$\cD(\cH(\F_q))$, where $l|q$ and $\cD(\cH(\F_q))=\text{Im}(\cH^{sc}(\F_q)\to\cH(\F_q))$.

This tells us that $H$ has a normal subgroup, $N$ say, where $N\cong\cD(\cH(\F_q))$. 
\begin{claim} The degree $[\F_q:\F_l]$ divides $[k:\F_l]$.
\end{claim}
\begin{proof}
We can consider $r:\cH^{sc}(\F_q)\to\cH(\F_q)\isoto N \into H \into \PGL_n(k)$ as a projective 
representation of the abstract group $\cH^{sc}(\F_q)$. We can then extend the coefficients of
this representation to $\bar{\F}_l$. Applying Theorem \ref{steinberg-rep-thm}, we see that,
after extending coefficients in this way, it can be constructed 
as the restriction to $\F_q$ points of a product of Frobenius twists of algebraic representations,
as described in the statement of that theorem.
But considering representations of this form, we immediately see that if $[\F_q:\F_l]\nmid[k:\F_l]$, 
we will not be able to conjugate our $\bar{\F}_l$ representation to a representation defined over $k$,
which is a contradiction, because our $\bar{\F}_l$ representation came from one defined over $k$.
\end{proof}
Since, by assumption, $l\nmid [k:\F_l]$ we see that $l\nmid [\F_q:\F_l]$, and hence (using 
the complete description of the outer automorphism groups of groups of Lie type given on 
page xv of \cite{conway-atlas}) that $\operatorname{Out} N$ has order prime to $N$. Thus
$H/N$ has order prime to $l$. Let $q:\GL_n(k)\onto\PGL_n(k)$ denote the quotient map,
and let $\Gamma_2=\Gamma=q^{-1}(\phi(H))$ and $\Gamma_1=q^{-1}(\phi(N))$. Then 
we see that $\Gamma_1\triangleleft \Gamma_2$ with $\#(\Gamma_2/\Gamma_1)=\#(H/N)$ 
prime to $l$. So it suffices to prove that $\Gamma_1$ is sturdy (since then $\Gamma$ is sturdy,
so we are done by the fourth alternative in the statement of the proposition).

But $\text{Im}(\cH^{sc}(\F_q)\to\cH(\F_q))$ is the simple group $\cD(\cH(\F_q))$,
so $\Gamma_1$ is immediately seen to be very sturdy, taking in the definition $m=1$,
$k_1=\F_q$, $\cG=\cH^{sc}$, and $r_1:\cD(\cH(\F_q))\isoto N\into \PGL_n(k)$.
\end{itemize}
This completes the proof of the proposition. (The last sentence of the proposition may be 
straightforwardly verified inductively.)
\end{proof}

We can now reformulate our proposition in the language of representations.
\begin{prop} \label{main prob, lang of reps}
For each positive integer $n$, there is an integer $A_n$ with the following
property. Let $l>A_n$ be a prime, $k/\F_l$ a finite extension with $l\nmid[k:\F_l]$,
 $\Gamma_0$ a
group, and $r:\Gamma_0\to\GL_n(k)$ a representation. For convenience of notation, let us
choose a number field $L$ and prime $\lambda$ of $L$ such that 
$\bigO_L/\lambda \bigO_L=k$.
Then one of the following must occur:
\begin{enumerate}
\item $r$ does not act absolutely irreducibly on $k^n$,
\item there is a proper subgroup $\Gamma_0'<\Gamma_0$ and representation 
$r':\Gamma_0'\to \GL_m(k)$ such that $r=\Ind_{\Gamma_0'}^{\Gamma_0} r'$,
\item there are representations 
$r_1:\Gamma_0\to\GL_m(\bigO_L)$ and
$r_2:\Gamma_0\to\GL_{m'}(k)$ with open kernels and with $m>1$ such that $r= \bar{r}_1\otimes r_2$, or
\item $k^\times\,r(\Gamma_0)$ is a sturdy subgroup.
\end{enumerate}
\end{prop}
Furthermore, in case (3) the order of the image of $r_1$ can be bounded in terms
of $n$.
\begin{proof}
We will take the constants $A_n$ to be the same constants as in the previous 
proposition, and will show that the property we require holds with this choice 
of the $A_n$. To see this, we apply the previous proposition to 
$r(\Gamma_0) k^\times<\GL_n(k)$. We deduce that one of the four possibilities (1)--(4)
in the conclusion of that proposition must hold. We split the remainder of our proof
into cases, according to which of the alternatives hold
\begin{itemize}
\item {\sl Case 1: $k^\times\, r(\Gamma_0)$ does not act absolutely irreducibly}. In this case
we immediately see that the first alternative of the present Proposition holds.
\item {\sl Case 2: $k^\times\,r(\Gamma_0)\subset G_{V_1\oplus\dots\oplus V_m}$, for some direct
sum decomposition $V\cong V_1\oplus\dots\oplus V_m$.} In this case,
for each $\gamma\in\Gamma_0$, $r(\gamma)$ must send each $V_i$ into some
other $V_i$, and indeed will determine in this way a permutation of the $V_i$.
Thus $\Gamma_0$ acts on the set $\{V_i |1\leq i\leq m\}$. If this permutation action
is not transitive, then $\bigoplus_{V_i \in \bigO_{\Gamma_0} V_1} V_i$  is a nontrivial
$\Gamma_0$ submodule of $V$, so $\Gamma_0$ acts irreducibly, and see that 
the first alternative of the present proposition holds. So we assume the action is 
transitive. Then let $\Gamma_0'$ be the stabilizer of $V_1$. $r|_{\Gamma_0'}$ sends 
$V_1$ to itself, and we then get a representation $r':\Gamma_0'\to\GL(V_1)$. 
We then see that $r=\Ind_{\Gamma_0'}^{\Gamma} r'$.
\item {\sl Case 3: the projective representation factors through a tensor product, and 
the image on one tensor factor lifts to a characteristic zero representation.} Then see
that the other tensor factor in fact lifts from being a projective representation to an 
ordinary representation. This gives us what we need. (The bound on the size of $G$
gives us the bound on the size of the image of $r_1$.)
\item {\sl Case 4: $k^\times\,r(\Gamma)$ is sturdy.} In this case we immediately have the fourth
alternative of the present Proposition.
\end{itemize}
This completes the proof.
\end{proof}

Finally, we can deduce the main theorem (Theorem \ref{main-thm} in the introduction).
\begin{thm} \label{main-theorem again}
For each pair of positive integers $n$ and $M$, there is an 
integer $C(M,n)$ with the following property. Let $l>C(M,n)$ be a prime, $k/\F_l$ a finite 
extension with $l\nmid[k:\F_l]$, 
$\Gamma_0$ a group, and $r:\Gamma_0\to\GL_n(k)$ a representation. 
For convenience of notation, let us choose a number field $L$ and prime $\lambda$ 
of $L$ such that $\bigO_L/\lambda \bigO_L=k$. Suppose that the image of $r$ is not 
$M$-big. Then one of the following must hold:
\begin{enumerate}
\item $r$ does not act absolutely irreducibly on $k^n$,
\item there is a proper subgroup $\Gamma_0'<\Gamma_0$ and representation 
$r':\Gamma_0'\to \GL_m(k)$ such that $r=\Ind_{\Gamma_0'}^{\Gamma_0} r'$, or
\item there are representations 
$r_1:\Gamma_0\to\GL_m(\bigO_L)$ and
$r_2:\Gamma_0\to\GL_{m'}(k)$ with open kernels and with $m>1$ such that $r= \bar{r}_1\otimes r_2$.
\end{enumerate}
Furthermore, in case (3) the order of the image of $r_1$ can be bounded in terms
of $n$.
\end{thm}
\begin{proof} This follows immediately from Proposition \ref{main prob, lang of reps},
Proposition \ref{bigness and scalars}, and Corollary \ref{sturdy subgroups are big}.
\end{proof}

We also now prove the Lemma from the introduction asserting that the third option
above never occurs for the residual representations of regular crystalline Galois 
representations where $l$ is large compared to the weight.

\begin{proof}[Proof of Lemma \ref{regular weights lemma}]
Suppose that $n$, $N$, $F$ and $a$ are given as in the statement of the lemma.
We choose $D(n,N,a)=(3(\max_{\tau,i} a_{\tau,i} - \min_{\tau,i} a_{\tau,i})+2)N!$. 

Now suppose that $l$, $L$, $\rho$, and $\iota$ are as in the lemma. Remember
that $l$ has been chosen not to ramify in $F$, so that we may apply 
Fontaine-Laffaille theory.
Choose $v$ 
to be some place above $l$ in $F$. There is some unramified local extension 
$K/F_v$ such that the semisimplification of $\rho|_{I_K}$ breaks up as a direct sum of 
characters, say $\chi_1,\dots,\chi_n$. Let us choose an embedding 
$\sigma:K\into \bar{\Q}_l$.
We may then write
$$\chi_i=\omega_{[L:\Q_l]}^{c_{i,0}+c_{i,1}l+c_{i,2}l^2+\dots+c_{i,[L:\Q_l]-1}l^{[L:\Q_l]-1}}$$
where the $c_{i,j}$ are integers and $\omega_{[L:\Q_l]}$ is Serre's fundamental
character of niveau $[L:\Q_l]$, and moreover
$$\{c_{1,t},c_{2,t},\dots,c_{n,t}\}=a_{\iota\circ\Frob^t\circ\sigma|_L}$$

Write $k'$ for the residue field of $L$.
We may choose an element $\gamma\in I_K\subset G_F$ which maps to a generator
$g\in (k')^\times$ under $\omega_{[L:\Q_l]}$. 

Now, suppose for contradiction that $\bar{\rho}$ did break up as a tensor product
$\bar{\rho}'\otimes r''$. Choose  distinct
$\bar{k}$ eigenvectors of $\bar{\rho}'(\gamma)$, say $e_{\bar{\rho}',1}$ and $e_{\bar{\rho}',2}$.
The ratio of their eigenvalues must be $\alpha$, an $r$th root of unity for some $r\leq N!$.
(This is since all eigenvalues of $\rho(\gamma)$ are roots of unity of order dividing
the order of the image of $\rho$.)
Choose also a $\bar{k}$ eigenvector of $r''(\gamma)$, say $e_{r''}$. Then
$e_{\bar{\rho}',1}\otimes e_{r''}$ and $e_{\bar{\rho}',2} \otimes e_{r''}$ are eigenvectors
of $\bar{\rho}(\gamma)=(\bar{\rho}'\otimes r'')(\gamma)$. The ratio of their eigenvalues
is again $\alpha$; but on the other hand we see that the ratio of their eigenvalues is 
$\chi_i(\gamma)/\chi_j(\gamma)$ for some $i,j$, $1\leq i,j\leq n$, with $i,j$ distinct.

That is
\begin{align*}
\alpha &=\chi_i(\gamma)/\chi_j(\gamma) = g^{c_{i,0}+c_{i,1}l+\dots+c_{i,[L:\Q_l]-1}l^{[L:\Q_l]-1}}/
g^{c_{j,0}+c_{j,1}l+\dots+c_{j,[L:\Q_l]-1}l^{[L:\Q_l]-1}} \\
&= g^{(c_{i,0}-c_{j,0})+l(c_{i,1}-c_{j,1})+l^2(c_{i,2}-c_{j,2})+\dots+l^{[L:\Q_l]-1}(c_{i,[L:\Q_l]-1}-c_{j,[L:\Q_l]-1})}
\end{align*}
So, since $\alpha$ is an $r$th root of unity and $g$ is a generator,
\begin{equation}\label{horrible eqn}
r\left[(c_{i,0}-c_{j,0})+l(c_{i,1}-c_{j,1})+\dots+l^{[L:\Q_l]-1}(c_{i,[L:\Q_l]-1}-c_{j,[L:\Q_l]-1})\right]\equiv 0\quad \text{mod $l^{[L:\Q_l]} -1$}
\end{equation}

On the other hand, we have the following elementary fact, whose proof is an exercise:
\begin{claim} Since $r<N!$, and $l>(3(\max_{\tau,i} d_{\tau,i} - \min_{\tau,i} d_{\tau,i})+2)N!$, 
if we write $S$ for the set $\Z\cap[\min_{\tau,i} d_{\tau,i} - \max_{\tau,i} d_{\tau,i},
\max_{\tau,i} d_{\tau,i} - \min_{\tau,i} d_{\tau,i}]$, then the map
\begin{align*}
S\times\dots\times S&\to \Z/(l^{[L:\Q_l]} -1)\Z\\
(b_0,\dots,b_{[L:\Q_l] -1})&\mapsto (b_0+lb_1+\dots+l^{[L:\Q_l]-1}b_{[L:\Q_l]})r
\end{align*}
is injective.
\end{claim}
Thus, since each $c_{i,t}$ is an $a_{\tau,i}$ for some $\tau$ and $i$, we see that 
$c_{i,t}-c_{j,t}\in S$ for each $i,j,t$, so the claim and eq \ref{horrible eqn} tell us:
$$c_{i,0}-c_{j,0}=c_{i,1}-c_{j,1}=\dots=c_{i,[L:\F_l]-1}-c_{j,[L:\F_l]-1}=0$$
But the fact that $c_{i,0}=c_{j,0}$ tells us some two members of $a_{\iota\circ\sigma|_L}$
are equal, which is contrary to hypothesis. The lemma is therefore proved.
\end{proof}

\subsection{}
Given Theorem \ref{main-theorem again}, it is perhaps natural to ask how often we can expect 
representations of the forms (1)--(3) to have big image. (That is, how often representations 
the theorem does not guarantee to have big image have big image nonetheless.) We will make 
only some very superficial remarks. First, all representations of type (1) will fail to be big,
since acting absolutely irreducibly is part of the definition of bigness.

Second, it seems that representations of the form (3) will `for the most part' fail to have big
image. For instance, an inspection of \cite{conway-atlas} shows that the vast majority
of characteristic 0 representations of finite simple nonabelian groups $G$ seem to have 
dimension $n$ much larger than the maximal order $m$ of an element in $G$, and in such a case
(since the $n$ roots of characteristic polynomials of the elements of $G$ acting via the 
representation must them be $k$th roots of unity where $k\leq m$) it seems perhaps a little
unlikely that one would be able to find an element in $G$ whose characteristic polynomial
has a simple root. Some rudimentary computer calculations seem to back up this intuition.

On the other hand, for very low dimensional representations, it seems that such reductions
of Artin representations do `generally' have big image. 
For instance, in \S4 of \cite{blggord} it is shown for $l>5$ that
any two dimensional representation which acts absolutely irreducibly and is not induced
is in fact 2-big, and
the author believes that any three dimensional representation which is not induced and
acts absolutely irreducibly 
is both 2-big and 3-big. (For four dimensional representations, things rapidly become less
nice: if we let $\rho: (2.A_4)\times (2.A_4)\to \GL_4(\Qbar)$ by tensoring together two copies
of the standard two-dimensional representation of $2.A_4$, it is not hard to see that
the image of $\bar{\rho}$ fails to be big. Here $2.A_4$ is the binary tetrahedral group, the 
unique double cover of $A_4$.) 

This leaves representations of the form (2), about which we have less to say. The question
of whether an induced representation has big image seems to be rather delicate, and 
depends on precise details of the induced representation. It is certainly possible for such
a representation to fail to be big (for instance, let $K_1/\Q$ and $K_2/\Q$ be two quadratic
extensions, $k/\F_l$ a finite extension, $\theta_1:G_{K_1}\to k^\times$ and 
$\theta_2:G_{K_2}\to k^\times$ two characters, and 
$r=(\Ind_{G_{K_1}}^{G_\Q} \theta_1)\otimes(\Ind_{G_{K_1}}^{G_\Q} \theta_2)$; then it is relatively
straightforward to show that $r$ will fail to have big image). On the other hand, it seems
that many such induced representations will in fact have big image.

\bibliographystyle{amsalpha}
\bibliography{barnetlambgeegeraghty}

\end{document}